\documentclass{siamonline190516}
\usepackage{a4,latexsym,exscale,theorem,epsfig}
\usepackage{amssymb,psfrag,epsf,amsmath,verbatim,bbm,float}
\usepackage{listings}
\usepackage{color}
\usepackage{mathbbol}
\usepackage{enumerate}
\newtheorem{remark}[theorem]{Remark}
\newtheorem{remarks}[theorem]{Remarks}
\newtheorem{example}[theorem]{Example}
\usepackage{geometry}
\usepackage{accents}
\usepackage{algpseudocode}
\usepackage{todonotes}
\begin{document}
\newcommand {\eps} {\varepsilon}
\newcommand {\Z} {\mathbbm{Z}}
\newcommand {\R} {\mathbbm{R}}
\newcommand {\N} {\mathbbm{N}}
\newcommand {\C} {\mathbbm{C}}
\newcommand {\Q} {\mathbbm{Q}}
\newcommand {\T} {\mathbbm{T}}
\newcommand {\I} {\mathbbm{I}}
\newcommand {\dist} {{\rm{dist}}}
\newcommand {\cl}{\mathrm{cl}}
\newcommand {\PP} {\mathbbm{P}}
\newcommand {\ang} {\measuredangle}
\newcommand {\e} {{\rm{e}}}
\newcommand {\rank} {{\rm{rank}}}
\newcommand {\Span} {{\mathrm{span}}}
\newcommand {\card} {{\rm{card}}}
\newcommand {\ED} {\mathrm{ED}}
\newcommand {\cA} {\mathcal{A}}
\newcommand {\cO} {\mathcal{O}}
\newcommand {\cF} {\mathcal{F}}
\newcommand {\cJ} {\mathcal{J}}
\newcommand {\cC} {\mathcal{C}}
\newcommand {\cL} {\mathcal{L}}
\newcommand {\cN} {\mathcal{N}}
\newcommand {\cV} {\mathcal{V}}
\newcommand {\cG} {\mathcal{G}}
\newcommand {\cB} {\mathcal{B}}
\newcommand {\cD} {\mathcal{D}}
\newcommand {\cP} {\mathcal{P}}
\newcommand {\cQ} {\mathcal{Q}}
\newcommand {\cW} {\mathcal{W}}
\newcommand {\cT} {\mathcal{T}}
\newcommand {\cI} {\mathcal{I}}
\newcommand {\Sn}[1] {\mathcal{S}^{#1}}
\newcommand {\range} {\mathcal{R}}
\newcommand {\kernel} {\mathcal{N}}
\newcommand{\one}{\mathbb{1}}
\renewcommand{\thefootnote}{\fnsymbol{footnote}}
\newcommand{\rle}{\rotatebox[origin=c]{-90}{$\le$}}
\newcommand{\rl}{\rotatebox[origin=c]{-90}{$<$}}
\newcommand{\rg}{\rotatebox[origin=c]{-90}{$=$}}
 

\title{\bf Smoothness properties of principal angles between subspaces \\
with applications to angular values of dynamical systems}

\author{Wolf-J\"urgen Beyn\footnotemark[1]\qquad
  Thorsten H\"uls\footnotemark[1]
}
\footnotetext[1]{Department of Mathematics, Bielefeld University,  
33501 Bielefeld, Germany \\
\texttt{beyn@math.uni-bielefeld.de}, \texttt{huels@math.uni-bielefeld.de}}

\maketitle


 \begin{abstract}
   In this work we provide detailed estimates of maximal principal angles between
   subspaces and we analyze their smoothness for smoothly varying subspaces.
   This leads to a new definition  of angular values for linear dynamical
   systems in continuous time. We derive some of their properties
   complementary to the theory of angular values  developed in  \cite{BeFrHu20},  \cite{BeHu22} for discrete time systems. The estimates are further employed
   to establish upper semicontinuity of angular values
   for some parametric model examples of discrete and continuous type.
   \end{abstract}

\begin{keywords}
  principal angles between subspaces, invariance and smoothness,
  angular values of dynamical systems, upper semicontinuity.
\end{keywords}

\begin{AMS}
37C05, 37E45, 34D09, 65Q10. 
\end{AMS}


\section{Introduction}
\label{sec1}
Principal angles between linear subspaces of a Euclidean space form a
standard tool in numerical linear algebra which can be computed
via a singular value decomposition (\cite{BG73, Dr2000}, \cite[Ch.6.4]{GvL2013}). The maximal principal angle is a  measure of the distance
between two elements in the Grassmann manifold of all subspaces of a fixed
dimension (it is indeed a metric; see Proposition \ref{prop2g}). This motivates its usefulness for various areas  of application, such as optimization
\cite{Mohammadi2014} signal processing \cite{ES2005, LCSJ17}, and finite element methods \cite{LSS22}.

In two previous papers \cite{BeFrHu20, BeHu22} the authors (in
\cite{BeFrHu20} jointly with G. Froyland) introduced the new notion of angular
values for subspaces of arbitrary dimensions in  a linear time discrete dynamical system.
In essence, the angular value of dimension
$s$ measures the maximal average rotation in terms of principal angles which
an arbitrary subspace of dimension $s$ experiences through the dynamics of
the given system. In the discrete time setting, it turned out that precise
estimates and continuity properties are essential for deriving
reduction theorems \cite[Section 3]{BeHu22} and explicit formulas for
autonomous systems \cite[Section 5,6]{BeFrHu20}. In particular, the
algorithm set up in \cite[Section 4]{BeHu22} is based on principal
angles and the reduction principle.

The first goal of this paper is to sharpen Lipschitz estimates of the
maximal principal angle and to show its smoothness when the underlying
subspace moves smoothly with respect to a parameter; see Lemma \ref{lem3:3}
and Theorem \ref{lemderivang}. Let us note that this particular derivative
can be obtained without any assumption on the leading singular value
or on higher regularity.
This is in contrast to the general problem of  smooth singular values which has
been studied extensively in the literature; see \cite{bbmn91, diei99, de06}.

In Section \ref{sec4} we use the formula for the derivative to define angular
values of $4$ types for nonautonomous linear systems in continuous time.
Then we investigate the invariance of angular values under asymptotically
constant kinematic transformations (Propositions \ref{prop1.1} and \ref{propcont})
in both discrete and continuous time. In the continuous autonomous case
this allows us to reduce the computation of angular values to systems with
matrices in real quasitriangular Schur form. If the Schur form contains several $2\times2$ blocks
with complex eigenvalues then the explicit formula (see Proposition \ref{prop4:autonomous})
leads to an integral expression which relies on the rational independence of the
frequencies and which involves Birkhoff's ergodic theorem (see e.g.\ \cite[Theorem 2.2]{Ba12}).

Finally, we consider the perturbation sensitivity of angular values,
in particular continuity with respect to parameters.
In \cite[eq.(6.8)]{BeFrHu20} we found the surprising fact that lower semicontinuity
fails for angular values in general, even for a discrete autonomous system
with a parameter dependent $2\times 2$-matrix. Nevertheless, we show that upper semicontinuity holds for this example which solves an open problem from \cite[Outlook]{BeFrHu20}. Moreover, we prove in Section \ref{sec4.4} that upper semicontinuity holds for a continuous
time system with a $4\times 4$-matrix which has two complex eigenvalues.
In this case critical points where lower semicontinuity fails are determined
by rational ratios of both frequencies.


\section{Principal angles and metrics on the Grassmannian}\label{sec2}

In this section we collect basic properties of principal angles between
subspaces. We show that the maximal principal angle provides a metric
on the Grassmannian and compare it with other common metrics. Moreover,
we study the invariance of principal angles under linear transformations.

\subsection{Definition of principal angles and elementary properties} \label{sec2.1}
The following definition is taken from \cite[Ch.6.4.3]{GvL2013}.
\begin{definition} \label{def1}
Let $V,W$ be subspaces of $\R^d$ of dimension $s$.
Then the principal angles $0\le \phi_1 \le \ldots \le \phi_{s}\le
\frac{\pi}{2}$ and associated principal vectors $v_j\in V$, $w_j \in W$
 are defined recursively for $j=1,\ldots,s$ by
\begin{equation*} \label{eq1.0}
  \cos(\phi_j)= \max_{\substack{v\in V,
                         \|v\|=1 \\ v^{\top}v_{\ell}  =0,
      \ell=1,\ldots,j-1}}\
\max_{\substack{w\in W, \|w\|=1 \\ w^{\top}
                    w_{\ell} =0,\ell=1,\ldots,j-1}} v^{\top}w
=v_j^{\top} w_j.
\end{equation*}
\end{definition}
We write $\ang(V,W)= \phi_s$ for the largest principal angle and in case
$s=1$ set $\ang(v,w):= \ang(\mathrm{span}(v),\mathrm{span}(w))$ for
$v,w \in \R^d$, $v,w \neq 0$.
Further, recall from
\cite[Prop.2.3]{BeFrHu20} an alternative variational characterization  and
from \cite[Ch.6.4.3]{GvL2013} an algorithm for computing principal angles via singular values.

 \begin{proposition}\label{Lemma2}
   Let $V, W\subseteq \R^d$ be two $s$-dimensional subspaces.
   \begin{itemize}
     \item[(i)]
  The principal angles and principal vectors satisfy for $j=s,\ldots,1$
  \begin{equation*} \label{A2}
    \cos(\phi_j)=
    \min_{\substack{v\in V, \|v\|=1
            \\v^{\top}v_{\ell} =0,\ell=j+1,\ldots,s }}\
\max_{\substack{w\in W, \|w\|=1 \\w^{\top} w_{\ell} =0,\ell=j+1,\ldots,s}} v^{\top}w
=v_j^{\top} w_j.
  \end{equation*}
  In particular, the following max-min principle holds
  \begin{equation}\label{A1}
    \ang(V,W) = 
    \max_{\substack{v\in V \\ v\neq 0} }\min_{\substack{w\in W \\ w \neq 0}} \ang(v,w)
  = \arccos\big(\min_{\substack{v\in V\\\|v\|=1}} \max_{\substack{
      w\in W\\\|w\|=1}} v^{\top} w\big).
  \end{equation}
\item[(ii)]
  Choose $P, Q \in \R^{d,s}$ such that $P^{\top}P=Q^{\top}Q=I_s$,
  $V=\range(P)$, $W=\range(Q)$ and consider the SVD
\begin{equation} \label{eq1.b}
P^{\top}Q = Y \Sigma Z^{\top}, \quad Y,Z,
\Sigma=\mathrm{diag}(\sigma_1,\ldots,\sigma_{s}) \in \R^{s,s},
\quad Y^{\top}Y=I_s=Z^{\top}Z,
\end{equation}
where $\sigma_1 \ge \cdots \ge \sigma_s>0$. 
Then the principal angles of $V$ and $W$  satisfy 
$\sigma_j = \cos(\phi_j)$ for $j=1,\ldots, s$
with principal vectors given by
\begin{equation*} \label{eq1.d}
PY = \begin{pmatrix}  v_1 & \cdots & v_{s} \end{pmatrix}, \quad
QZ = \begin{pmatrix}  w_1 & \cdots & w_{s} \end{pmatrix}.
 \end{equation*}
\end{itemize}
\end{proposition}

\subsection{Some metrics on the Grassmannian}
\label{sec2.2}

In this section we consider the Grassmann manifold (see \cite[Ch.6.4.3]{GvL2013})
\begin{equation*} \label{eq1.4}
  \mathcal{G}(s,d) = \{ V \subseteq \R^d \; \text{is a subspace of dimension}
  \; s \}
\end{equation*}
and study its various metrics.
\begin{proposition} \label{prop2g}
  The Grassmannian $\mathcal{G}(s,d)$ is a compact smooth manifold of
  dimension $s(d-s)$. Two metrics on $\cG(s,d)$ are given
for $ V,W \in \mathcal{G}(s,d)$  by
\begin{equation} \label{eq1:met12}
     d_1(V,W)  = \ang(V,W), \qquad
    d_2(V,W) =  \| P_V - P_W \|,
     \end{equation}
  where $P_V,P_W$ are the orthogonal projections onto $V$ and $W$, respectively,
  and $\|\cdot \|$ is the spectral norm.
  Both metrics are equivalent and related
  for $ V,W \in \mathcal{G}(s,d)$ by
  \begin{equation} \label{eq1:relmet}
    \begin{aligned}
  d_2(V,W) &= \sin (\ang(V,W)),\\
   \frac{2}{\pi} \ang(V,W)& \le d_2(V,W)\le \ang(V,W).
    \end{aligned}
    \end{equation}
 \end{proposition}
\begin{proof} The results for the metric $d_2$ can be found in
  \cite[Ch.2.5.3,6.4.3]{GvL2013} with the estimate in \eqref{eq1:relmet} 
  taken  from \cite[Prop.2.5]{BeFrHu20}. The definiteness and symmetry
  of $d_1$ follow directly from Proposition \ref{Lemma2} (ii).
  For the triangle inequality we quoted \cite[Theorem 3]{js96}
  in \cite[Prop.2.5]{BeFrHu20}, but realized that 
  this reference uses an angle $\theta$ between subspaces related  to principal angles
  by $\cos(\theta)=\cos(\phi_1)\cdots \cos(\phi_s)$; see \cite[Theorem 5]{js96}.
  We did not find a reference to the triangle inequality
  for $d_1$ elsewhere, so we provide a proof here for completeness.
   Let us first consider the case $s=1$.
  Let $u,v,w \in \R^d$ satisfy $\|u\|=\|v\|=\|w\|=1$ and let 
  $B= \begin{pmatrix} u & w & v \end{pmatrix} \in \R^{d,3}$, so that
  $\range(B) = \mathrm{span}(u,w,v)$. Further, by suitable sign changes of
  $v$ and $w$, we can arrange  $u^{\top}v \ge 0$ and $u^{\top}w \ge 0$
  (but not necessarily $v^{\top}w \ge 0$ !).
  Then take a short QR-decomposition of $B$ i.e.
  \begin{align*}
    B=QR, \quad Q\in \R^{d,3},\ Q^{\top}Q=I_3, \quad R=
    \begin{pmatrix}\tilde{u}& \tilde{w} & \tilde{v} \end{pmatrix}
    \in \R^{3,3},
  \end{align*}
  where $R$ is upper triangular and has nonnegative diagonal elements.
  Since $R^{\top}R=B^{\top}B$, all inner products are conserved and
  we can assume
  \begin{align*}
    \tilde{u}=\begin{pmatrix} 1 \\ 0 \\ 0 \end{pmatrix},\quad
    \tilde{w}= \begin{pmatrix} \cos(\alpha) \\ \sin(\alpha) \\ 0 \end{pmatrix},\quad
    \tilde{v}= r \begin{pmatrix} \cos(\beta) \\ \sin(\beta) \\0 \end{pmatrix}
    + \begin{pmatrix} 0 \\ 0 \\ v_3 \end{pmatrix},
  \end{align*}
  where  $0 \le \alpha \le \frac{\pi}{2}$, $|\beta| \le \frac{\pi}{2}$,
  $r \ge 0$, 
  and $r^2 + v_3^2 =1$. Recall the relations $0 \le u^{\top}w= \tilde{u}^{\top}\tilde{w}
  = \cos(\alpha)$, $0 \le u^{\top}v=  \tilde{u}^{\top}\tilde{v}  =r \cos(\beta)$,
  and $0 \le R_{22}=\sin(\alpha)$. By the invariance of the angle under
  orthogonal transformations (cf.\ Section \ref{sec2.3}) we have
  \begin{align*}
    \ang(u,v)=\ang(\tilde{u},\tilde{v}), \quad \ang(u,w)=
    \ang(\tilde{u},\tilde{w}),\quad \ang(w,v)=\ang(\tilde{w},\tilde{v}).
  \end{align*}
  Hence it suffices to prove
  \begin{align*}
    \ang(\tilde{u},\tilde{w})& \le \ang(\tilde{u},\tilde{v})+
    \ang(\tilde{v},\tilde{w}).
  \end{align*}
  Since $0 \le r \le 1$, $\cos(\beta)\ge 0$ and $\arccos$ is monotone decreasing
  we find
  \begin{align*}
    \ang(\tilde{u},\tilde{w}) & = \alpha, \quad
    \ang(\tilde{u},\tilde{v})= \arccos(r \cos(\beta))
    \ge \arccos(\cos(\beta))= |\beta|.
   \end{align*}
  The value of
  $  \ang(\tilde{v},\tilde{w})= \arccos(r| \cos(\alpha -\beta)|)$
  depends on the sign of $\cos(\alpha -\beta)$, which is positive for
  $|\alpha-\beta| \le \frac{\pi}{2}$ and negative otherwise.
  Since $\alpha-\beta \ge - \frac{\pi}{2}$ the second case occurs for
  $\frac{\pi}{2} < \alpha - \beta \le \pi $ where we have
   $-\cos(\alpha -\beta)=\cos(\pi-(\alpha- \beta))$.
  Summarizing, we obtain
  \begin{equation*} \label{eq1.5}
    \begin{aligned}
      \ang(\tilde{v},\tilde{w}) & = \begin{cases} \arccos(r\cos(\alpha -\beta)),
        & |\alpha-\beta| \le \frac{\pi}{2}, \\
        \arccos(r \cos(\pi-\alpha + \beta)), & \frac{\pi}{2} < \alpha - \beta,
      \end{cases}\\
      & \ge \begin{cases} |\alpha- \beta|, & |\alpha-\beta| \le \frac{\pi}{2},\\
        \pi - \alpha + \beta, & \beta < \alpha - \frac{\pi}{2} \le 0.
        \end{cases}
    \end{aligned}
  \end{equation*}
  We find in the first case
  \begin{align*}
    \ang(\tilde{u},\tilde{w}) & = \alpha \le |\beta| + |\alpha - \beta|
    \le \ang(\tilde{u},\tilde{v})+\ang(\tilde{v},\tilde{w})
  \end{align*}
  and in  the second case
  \begin{align*}
    \ang(\tilde{u},\tilde{w}) & = \alpha \le \pi - \alpha =
    |\beta|+ \pi -\alpha + \beta \le \ang(\tilde{u},\tilde{v})+\ang(\tilde{v},\tilde{w}).
  \end{align*}
  This proves the triangle inequality for $s=1$.
  For $s \ge 1$ we use the representation \eqref{A1} from the max-min principle.
  For ease of notation we set $U_0=U \setminus\{0\}$, $V_0=V \setminus\{0\}$,
  $W_0=W \setminus\{0\}$. Using the estimate from case $s=1$ we conclude
  \begin{align*}
    \ang(u,w) & \le \ang(u,v) + \ang(v,w), \quad \forall u\in U_0, v\in V_0, w \in W_0, \\
    \min_{w\in W_0}\ang(u,w) & \le \ang(u,v) + \min_{w \in W_0}\ang(v,w), \quad \forall u\in U_0, v\in V_0,\\
    \min_{w\in W_0}\ang(u,w) & \le\min_{v \in V_0}\left[\ang(u,v) + \min_{w \in W_0}\ang(v,w)\right]\\ & \le \min_{v \in V_0}\ang(u,v) + \max_{v\in V_0}\min_{w \in W_0}\ang(v,w), \quad \forall u\in U_0,\\
      \max_{u \in U_0} \min_{w\in W_0}\ang(u,w) & \le \max_{u \in U_0}\min_{v \in V_0}\ang(u,v) + \max_{v\in V_0}\min_{w \in W_0}\ang(v,w),\\
        \ang(U,W) & \le \ang(U,V) + \ang(V,W).
  \end{align*}
\end{proof}

  In the following we consider the Stiefel manifold of orthogonal $d \times s$ matrices which coincides with the orthogonal group $O(s)$ for $s=d$:
  \begin{align*}
    \mathrm{St}(s,d)= \{P \in \R^{d,s}:P^{\top}P=I_s\}, \quad O(s) =\mathrm{St}(s,s).
  \end{align*}
  The Grassmannian $\cG(s,d)$ can be identified with a quotient space $\mathrm{St}(s,d)/\sim$ of the Stiefel manifold as follows
  \begin{align*}
     P_1,P_2 \in\mathrm{St}(s,d):& \quad   P_1  \sim P_2 \Longleftrightarrow \exists \ Q\in O(s): P_1 = P_2Q, \\
   \mathrm{St}(s,d)/\sim \quad  \ni [P]_{\sim}  & \longleftrightarrow  V= \range(P) \in \cG(s,d).  \end{align*}
    There is an  interesting connection of principal angles and the metrics from Proposition \ref{prop2g} to the Procrustes problem  (\cite[Ch.6.4.1]{GvL2013}):
 given $P_1,P_2\in \R^{d,s}$, one minimizes
  \begin{equation} \label{procrustes}
   \Phi(Q)= \|P_1 - P_2 Q\|_{F}, \quad \text{with respect to } \quad Q\in O(s),
  \end{equation}
  where $\|P\|_{F}= (\mathrm{tr}(P^{\top}P))^{1/2}$, $P \in \R^{d,s}$  denotes the
  Frobenius norm.  The solution
  to \eqref{procrustes}  is given by (see \cite[Ch.6.4.1]{GvL2013})
  \begin{align*}
    \mathrm{argmin}(\Phi)=Z Y^{\top}, \quad \min \Phi=\big( \|P_1\|_F^2 +
    \|P_2\|_F^2-2 \mathrm{tr}(\Sigma)\big)^{1/2},
  \end{align*}
  where $Y,Z,\Sigma$ are defined by the SVD \eqref{eq1.b} for $P_1^{\top}P_2$. If
  $P_1,P_2\in \mathrm{St}(s,d)$  then we find
  \begin{align*}
    \min \Phi=\big(s + s - 2 \sum_{j=1}^s \cos(\phi_j)\big)^{1/2}=
   2 \big(\sum_{j=1}^s 
    \sin^2(\frac{\phi_j}{2})\big)^{1/2}.
  \end{align*}
  By the orthogonal invariance of the Frobenius norm we have
  \begin{equation} \label{Frobmin}
    \min_{Q\in O(s)}\|P_1 - P_2 Q\|_{F}=
    \min_{Q_1,Q_2 \in O(s)}\|P_1Q_1-P_2 Q_2 \|_F,
  \end{equation}
  so that
  \begin{equation} \label{Frobmin2}
    d_F(V_1,V_2)=  \min_{Q_1,Q_2 \in O(s)}\|P_1Q_1-P_2 Q_2 \|_F,
    \quad V_j= \range(P_j),\ P_j\in \mathrm{St}(s,d),\ j=1,2
    \end{equation}
  is well defined for $V_1,V_2 \in \cG(s,d)$. In fact, $d_F$ defines another
  metric on $\cG(s,d)$. While definiteness and symmetry are obvious, the
  triangle inequality is also easily seen. For $V_j=\range(P_j)$,
  $P_j\in \mathrm{St}(s,d)$, $j=1,2,3$ select $Q_3 \in O(s)$ with
  $d_F(V_1,V_3)=\|P_1-P_3Q_3\|_F$ and $Q_2 \in O(s)$ with
  $d_F(V_3,V_2)=\|P_3 - P_2 Q_2\|_F$. Then we conclude
  \begin{align*}
    d_F(V_1,V_2) & \le \|P_1- P_2Q_2Q_3\|_F \le
    \|P_1- P_3Q_3\|_F+\|P_3Q_3- P_2Q_2Q_3\|_F\\
    &=\|P_1- P_3Q_3\|_F+\|P_3- P_2Q_2\|_F= d_F(V_1,V_3) +d_F(V_3,V_2).
    \end{align*}
  Thus we have shown the following result.
  \begin{corollary} \label{cor1.5}
    A metric on the Grassmannian $\cG(s,d)$ is given by
    \begin{equation*} \label{Frobmetric}
      d_F(V,W) = 2 \big( \sum_{j=1}^s \sin^2(\frac{\phi_j}{2}) \big)^{1/2},
    \end{equation*}
    where $\phi_j,\ j=1,\ldots,s$ denote the principal angles between
    $V,W \in \cG(s,d)$.
    \end{corollary}
  From the proof above we observe that any orthogonally invariant norm
  will lead to a metric on the Grassmannian via \eqref{Frobmin} and
  \eqref{Frobmin2}. The next proposition determines this metric for
  the spectral norm.
  \begin{proposition} \label{propspecmin}
    Let $P_1,P_2\in \mathrm{St}(s,d)$ be given with SVD
    $P_1^{\top}P_2=Y \Sigma Z^{\top}$. Then the
    following holds for the spectral norm $\|\cdot\|$:
    \begin{equation} \label{minspecval}
      \min_{Q_1,Q_2 \in O(s)} \|P_1 Q_1 - P_2Q_2\|= \sqrt{2(1-\sigma_s)}=
      2 \sin(\frac{\phi_s}{2}),
    \end{equation}
    with the minimum  achieved for $Q_1=Y$, $Q_2=Z$. Moreover,
    \begin{equation*} \label{metricspec}
      d_{\sigma}(V_1,V_2)= 2 \sin\big(\frac{\phi_s}{2}\big)=
      2 \sin\big(\frac{\ang(V_1,V_2)}{2}\big)
      ,  \quad V_j\in \cG(s,d),\quad
      j=1,2,
    \end{equation*}
    defines a metric on the Grassmannian $\cG(s,d)$.
  \end{proposition}
  \begin{proof}
    By our comments above it suffices to show \eqref{minspecval}. Let us
    compute
    \begin{align*}
     & (P_1Q_1-P_2Q_2)^{\top}(P_1Q_1-P_2Q_2) = 2 I_s - Q_1^{\top}P_1^{\top}P_2 Q_2 -
      Q_2^{\top} P_2^{\top}P_1Q_1 \\
      & = 2 I_s - Q_1^{\top}Y \Sigma Z^{\top} Q_2- Q_2^{\top}Z \Sigma Y^{\top}Q_1=
      Q_1^{\top}Y\big(2 I_s - \Sigma Q - Q^{\top}\Sigma \big)Y^{\top}Q_1,
    \end{align*}
    where $Q:=Z^{\top}Q_2 Q_1^{\top}Y \in O(s)$. Together with Rayleigh's principle
    we obtain for the spectral norm and the unit vector
    $e^s=\begin{pmatrix}0& \cdots& 0& 1 \end{pmatrix}^{\top}$
    \begin{align*}
      \| P_1Q_1-P_2Q_2\|^2& = \lambda_{\max} \big(2 I_s - \Sigma Q - Q^{\top}\Sigma \big)\\
      & = \max_{\|x\|=1} x^{\top}\big(2 I_s - \Sigma Q - Q^{\top}\Sigma \big)x \ge
      (e^s)^\top\big(2 I_s - \Sigma Q - Q^{\top}\Sigma \big) e^s \\
      & = 2 - 2 (e^s)^ \top \Sigma Q e^s = 2- 2 \sigma_s (e^s)^{ \top}Q e^s \\
      & \ge 2 - 2 \sigma_s \|e^s\| \|Q\| \|e^s\| = 2(1 -\sigma_s).
    \end{align*}
    Obviously, the lower bound is achieved when setting $Q=I_s$ and $Q_2=Z$, $Q_1=Y$,
    for example. The formula $2(1-\cos(\phi_s))=4 \sin^2 \big(\frac{\phi_s}{2}\big)$
    concludes the proof.
  \end{proof}

  \subsection{Invariance of principal angles} \label{sec2.3}
  By Definition \ref{def1}, principal angles between subspaces are invariant
  under orthogonal transformations; i.e.\ $\phi_j(V,W) = \phi_j(QV,QW)$
  for $V,W\in \cG(s,d)$, $Q\in O(d)$. The following lemma shows that
  these transformations are  the only ones (up to scalar multiples) which have this
  property. The result justifies the statements made after \cite[Prop. 3.8]{BeFrHu20}.
  \begin{lemma} \label{lem1}
  Two matrices $Q_1,Q_2 \in \R^{d,s}$ satisfy
  \begin{equation} \label{e1.1}
     \ang(Q_1u,Q_2v) = \ang(u,v) \quad \forall u,v \in \R^s \setminus \{0\}
  \end{equation}
  if and only if there exist constants $c_1,c_2 \neq 0$ such that
  \begin{equation}\label{e1.1a}
    Q_1^{\top}Q_1 = c_1^2 I_s, \quad Q_2^{\top}Q_2 = c_2^2 I_s, \quad
    Q_1^{\top}Q_2 = c_1 c_2 I_s.
  \end{equation}
  In case $s=d$ this is equivalent to
  \begin{equation*} \label{e1.2}
    Q_j= c_j Q,\ j=1,2 \quad \text{ for some } c_1,c_2 \neq 0,\
   Q \in O(d).
  \end{equation*}
  \end{lemma}
\begin{proof} Recall $\cos(\ang(u,v))=
  \frac{|u^{\top}v|}{\|u\| \|v\|}$ by definition. Assuming \eqref{e1.1a},
  we have
  \begin{align*}
    \cos(\ang(Q_1u,Q_2v))& = \frac{|(Q_1u)^{\top}Q_2v|}{\|Q_1 u\| \|Q_2v\|}
    = \frac{|c_1 c_2| | u^{\top}v|}{|c_1 c_2| \| u\| \|v\|} = \cos(\ang(u,v)).
  \end{align*}
  Conversely, let \eqref{e1.1} be satisfied. Note that for every $v \in \R^s$,
  $v \neq 0$ and
  every $u\in v^{\perp}$ we have $u^{\top}v=0$, hence $0= (Q_1u)^{\top}Q_2v= u^{\top}(Q_1^{\top}Q_2v)$ by \eqref{e1.1}. Therefore 
  \begin{equation} \label{e1.3}
    Q_1^{\top}Q_2v\in (v^{\perp})^{\perp} =\mathrm{span}(v) \quad \forall
    v \in \R^s.
    \end{equation}
  Taking the unit vectors  $v=e^j,\ j=1,\ldots,s$
  we obtain $ Q_1^{\top}Q_2 e^j=r_j e^j$ for some $r_j \neq 0$ and thus $Q_1^{\top}Q_2 = \mathrm{diag}(r_1,\ldots,r_s)$. Now set $v= r_ie^i + r_je^ j$ for $i\neq j$ in \eqref{e1.3} and find
  $Q_1^{\top}Q_2 v = \rho v$ for some $\rho \neq 0$. This implies
  \begin{align*}
    \frac{r_j^2}{r_i^2} = \frac{(Q_1^{\top}Q_2v)_j}{(Q_1^{\top}Q_2v)_i}=
    \frac{\rho v_j}{\rho v_i} = \frac{r_j}{r_i},
    \end{align*}
  hence $r_j=r_i$. Thus we have shown $Q_1^{\top}Q_2= rI_s$ for some $r \neq 0$.
  For $u,v\in \R^s$ with $u^{\top}v \neq 0$ we infer from \eqref{e1.1}
  \begin{align*}
    \frac{\|Q_1u\|}{\|u\|}& = \frac{\|Q_1u\| \|v\|}{|u^{\top}v|} \cos(\ang(u,v))
    =\frac{\|Q_1u\| \|v\|}{|u^{\top}v|} \cos(\ang(Q_1u,Q_2v))\\
    & = \frac{|u^{\top} Q_1^{\top}Q_2v| \ \|v\|}{|u^{\top}v| \|Q_2v\|}=
    \frac{|r| \|v\|}{ \|Q_2v\|}.
  \end{align*}
  Take e.g.\ $v= e^1$ and find for $c_1 = \frac{|r|}{\|Q_2e^1\|} >0$ the equality
  \begin{align*}
    \|Q_1u\|= c_1 \|u\| \quad \forall u \in \R^s \text{ with } u_1 \neq 0.
  \end{align*}
  By continuity, this holds for all $u \in \R^s$. Therefore, $c_1^{-1}Q_1$ has
  orthonormal columns and $Q_1^{\top}Q_1 = c_1^2 I_s$.
  Similarly, we find a constant $c_2>0$ such that $c_2^{-1}Q_2$ has
  orthonormal columns and $Q_2^{\top}Q_2 = c_2^2 I_s$. From the definition
  of $c_1$ we obtain $c_1= \frac{|r|}{c_2}$. By reversing the sign of $c_2$
  if $r<0$, we can arrange $c_1 c_2 = r$. This leads to $Q_1^{\top}Q_2=c_1 c_2 I_s$.
  
  Finally, in case $s=d$ we set $Q=c_1^{-1}Q_1$ and obtain $Q_2 = r Q_1^{-\top}= \frac{r}{c_1} Q^{-\top}= c_2Q$.  \end{proof}


\section{Smoothness of principal angles}
\label{sec3}
The main goal of this section is to derive a  formula for the initial velocity
of the maximal principal angle when a given subspace starts to move. This  will
be essential for defining angular values of dynamical systems in continuous
time. First we prove some  bounds and Lipschitz estimates of the maximal principal
angle which are employed in Section \ref{sec4} to study the invariance of angular values 
 and to derive explicit formulas in the autonomous case.
\subsection{Boundedness and Lipschitz properties}
\label{sec3.1}
The following lemma \cite[Lemma 2.8]{BeFrHu20} provides an explicit angle bound for linear transformations.

\begin{lemma} \label{lem3:1} (Angle bound for linear maps)
  Let $S \in \mathrm{GL}(\R^d)$ and $\kappa=\|S^{-1}\| \|S\|$ be its
  condition number. Then the following estimate holds
  \begin{equation*} \label{anglebound}
    \ang(SV,SW) \le \pi \kappa(1+\kappa) \ang(V,W) \quad
      \forall \  V,W \in \mathcal{G}(s,d),\quad 1\le
        s\le d.
  \end{equation*}
\end{lemma}
In the next lemma we estimate the angle between a given subspace and its
image. It is the key to essentially all asymptotic estimates in \cite{BeFrHu20}.
We provide here an alternative proof to \cite[Lemma 2.6]{BeFrHu20} based on
elementary geometry for triangles.
\begin{lemma} (Angle estimate of subspaces for near identity maps) \label{lem3:2}
  \begin{enumerate}[(i)]
  \item For any two vectors $v,w \in \R^d$ with $\|v\| < \|w\|$ the following holds
    \begin{equation} \label{eq1.6}
              \tan^2\ang(v+w,w) \le \frac{\|v\|^2}{\|w\|^2 - \|v\|^2}.
          \end{equation}
  \item
    Let $V \in \mathcal{G}(s,d)$ and $S \in \R^{d,d}$ be such that
    for some $0 \le q <1$
    \begin{equation*} \label{eq1.7}
      \| (I_d-S)v\| \le q \| Sv \| \quad \forall \  v \in V.
    \end{equation*}
    Then $\dim(V)=\dim(SV)$ and the following estimate holds
    \begin{equation} \label{eq1.8}
      \ang(V,SV) \le \frac{q}{(1-q^2)^{1/2}}.
    \end{equation}
  \end{enumerate}
\end{lemma}
\begin{figure}[H] 
    \begin{center}
      \includegraphics[width=0.38\textwidth]{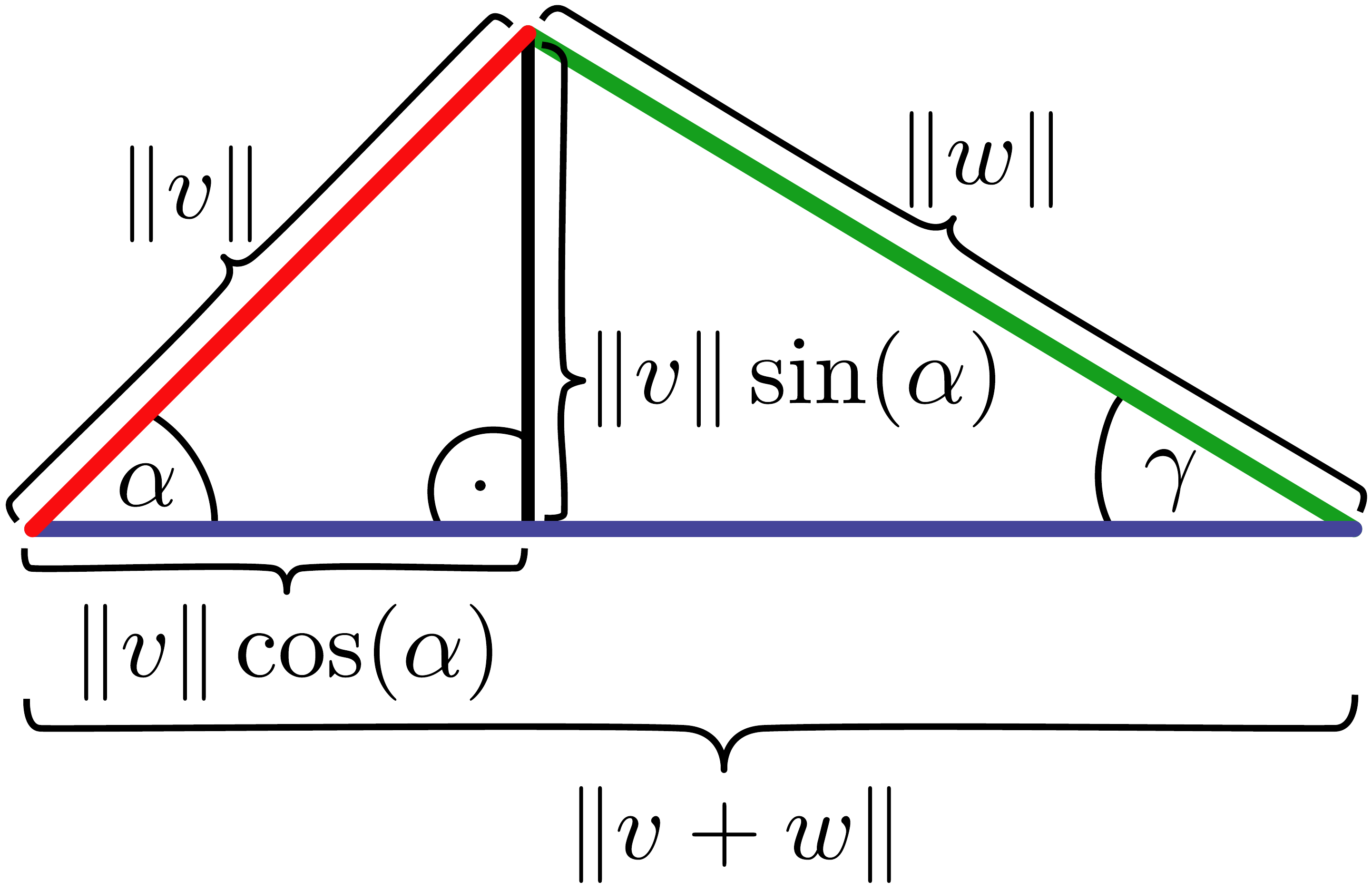}
    \end{center}
\caption{\label{Figeometry} Geometry of angle estimate \eqref{eq1.6}
  with $a=\|w\|$, $b=\|v+w\|$, $c=\|v\|$.} 
\end{figure}
\begin{proof}
Dropping a perpendicular from the top to the bottom edge shows
\begin{align*}
  \tan(\gamma)= \frac{c \sin(\alpha)}{b-c \cos(\alpha)}.
\end{align*}
Taking the square and using the cosine-Theorem $a^2=b^2+c^2 - 2 bc \cos(\alpha)$
leads to
\begin{align*}
  \tan^2(\gamma) & = \frac{c^2 \sin^2(\alpha)}{b^2 - 2 bc \cos(\alpha) + c^2 \cos^2(\alpha)}
  = \frac{c^2 \sin^2(\alpha)}{a^2 - c^2 + c^2 \cos^2(\alpha)}\\
  &=
  \frac{c^2 \sin^2(\alpha)}{a^2 - c^2 \sin^2(\alpha)}
  = \frac{c^2}{\tfrac{a^2}{\sin^2(\alpha)} - c^2 } \le \frac{c^2}{a^2 - c^2},
\end{align*}
which coincides with \eqref{eq1.6}. The estimate \eqref{eq1.8} 
then follows from \eqref{eq1.6} and the max-min principle
\eqref{A1}; see \cite[Lemma 2.6]{BeFrHu20}.
\end{proof}

As a last step we use Lemma \ref{lem3:2} to establish for fixed $V,W\in \cG(s,d)$ a global Lipschitz estimate  for the map
 $ S \mapsto \ang(SV,W)$ for $S$ near $I_d \in \R^{d,d}$.
\begin{lemma}(Lipschitz estimate of angles for general maps) \label{lem3:3}
  For all $V,W \in \mathcal{G}(s,d)$ and $S \in \R^{d,d}$
  with $SV \in \mathcal{G}(s,d)$ the following estimate holds:
  \begin{equation} 
  \label{Lipang}
  |\ang(SV,W)- \ang(V,W)| \le C \|S-I_d\|, \quad
  C= \tfrac{\pi}{2}+\left(\tfrac{\pi^2}{4}+1\right)^{1/2}.
  \end{equation}
\end{lemma}
\begin{proof}
  By the triangle inequality we have
     $|\ang(SV,W)- \ang(V,W)|  \le \ang(SV,V)$.
  Hence it suffices to show
  \begin{equation} \label{angPVV}
    \ang(SV,V) \le K \frac{ \pi}{2} \|S-I_d\|, \quad  K=1+ \left(1 + 4\pi^{-2}\right)^{1/2}.
  \end{equation}
  For $\|S-I_d\| \ge \frac{1}{K}$ this is trivial. Therefore, we consider
  $\|S-I_d\| \le \frac{1}{K}<1$. Then
  $S=I_d-(I_d-S)$ is invertible and satisfies
  \begin{align*}
    \|S^{-1}\|& \le \left(1- \frac{1}{K}\right)^{-1}= \frac{K}{K-1}.
  \end{align*}
  This implies for all $v\in V$
  \begin{align*}
    \|(I_d-S)v\|& \le \|I_d - S\| \|S^{-1}Sv\| \le
    \frac{K}{K-1}\|I_d -S\| \|Sv\|.
  \end{align*}
  Thus Lemma \ref{lem3:2} (ii) applies with $q=\frac{K}{K-1}\|I_d -S\|\le \frac{1}{K-1} <1$ and yields
  \begin{align*}
    \ang(V,SV)&\le \frac{\frac{K}{K-1} \|I_d -S\|}{\left(1 - (K-1)^{-2}\right)^{1/2}}
     = \frac{K \|I_d -S\|}{\left((K-1)^2 - 1\right)^{1/2}}.
  \end{align*}
 Since $K$ satisfies $\left((K-1)^2 - 1\right)^{1/2}=\frac{2}{\pi}$ 
  our assertion \eqref{angPVV} follows.
\end{proof}
\begin{remark}
  The same Lipschitz estimate \eqref{Lipang} holds for the metric $d_2$ from
  \eqref{eq1:met12} since by \eqref{eq1:relmet} and the triangle inequality
  \begin{align*}
    |d_2(SV,W) - d_2(V,W)| \le d_2(SV,V)= \sin(\ang(SV,V)) \le \ang(SV,V).
  \end{align*}
    \end{remark}

\subsection{Differentiability of principal angles}
\label{sec3.2}
The following theorem provides an explicit formula for the derivative
of the angle between smoothly varying subspaces.
\begin{theorem} \label{lemderivang} Assume $W\in C^1([0,T),\R^{d,s})$,
   $\rank(W(t))=s$  and let
    $V(t)=\range(W(t))$, $t \in [0,T)$ be the corresponding curve in
      $\cG(s,d)$.
  Then for every $\tau \in [0,T)$, the function
    $t \mapsto \ang(V(\tau),V(t))$
  is differentiable from the right at $\tau$ and the following
  formula holds
 \begin{equation} \label{eq3:derivformula}
   \frac{d}{d t}\ang(V(\tau),V(t))_{ |t= \tau+} =
   \|(I_d-P_{V(\tau)})\frac{dW}{dt}(\tau)  (W(\tau)^{\top}W(\tau))^{-1/2}\|,
 \end{equation}
 where $\|\cdot \|$ is the spectral norm in $\R^{d,s}$ and $P_{V(\tau)}$
 is the orthogonal projection onto $V(\tau)$.
\end{theorem}
\begin{proof} By a shift it is enough to consider $\tau=0$.
Since the matrices $W(t)$ have full rank an orthonormal basis of $V(t)$, $t \in [0,T)$ is given by the columns of
   \begin{align*}
    Q(t) =  W(t) (W(t)^{\top}W(t))^{-1/2}.
   \end{align*}
   For the proof we can further assume
   that $W_0:=W(0)$ has orthonormal columns so that $Q(0)=W_0$ and
   \eqref{eq3:derivformula} simplifies to
\begin{equation} \label{eq3:deriv0}
   \frac{d}{d t}\ang(V(0),V(t))_{ |t= 0+} =
   \|(I_d-W_0 W_0^{\top})\frac{dW}{dt}(0)\|.
 \end{equation}
In the general case, one applies \eqref{eq3:deriv0} to $\tilde{W}(t)=
W(t) (W(0)^{\top}W(0))^{-1/2}$ to obtain \eqref{eq3:derivformula}.
      In the following we denote the ordered principal angles between $V(0)$ and
   $V(t)$ by
   \begin{equation*} \label{orderangle}
    0\le \phi_1(t) \le \ldots \le
    \phi_s(t)=\ang(V(0),V(t)) \le \frac{\pi}{2}.
   \end{equation*}
   It is important that we do not make any continuity assumption
   with respect to $t \ge 0$.
   By Proposition \ref{Lemma2} the cosines of the principal angles are
   the singular values of $Q(t)^\top Q(0)$, i.e.\ the symmetric $s \times s$-matrix
   \begin{equation*}\label{defsymsquare}
     S(t) = (Q(t)^{\top}Q(0))^{\top}Q(t)^{\top}Q(0)= W_0^{\top} W(t) \big(W(t)^{\top}W(t)\big)^{-1} W(t)^{\top} W_0
  \end{equation*}
   has the eigenvalues
   \begin{equation*} \label{Seigenvalues}
     1 \ge \cos^2(\phi_1(t)) \ge \ldots \ge \cos^2(\phi_s(t)) \ge 0.
   \end{equation*}
   Therefore, $I_s-S(t)$ has the eigenvalues
   \begin{equation*} \label{1-Seigenvalues}
     0 \le \sin^2(\phi_1(t)) \le \ldots \le \sin^2(\phi_s(t)) \le 1.
   \end{equation*}
   Let us assume $W(\cdot)\in C^2$ so that we have a Taylor expansion
   \begin{equation} \label{expandW}
     W(t)= W_0 + t \dot{W}_0 + \frac{1}{2}t^2 \ddot{W}_0 + o(t^2).
   \end{equation}
   Below we will show the expansion
   \begin{equation} \label{expand1-S}
     I_s-S(t)= t^2 M_0^{\top}M_0 + o(t^2), \quad \text{where} \; M_0=(I_d - W_0W_0^{\top})\dot{W}_0.
   \end{equation}
   Let us first conclude \eqref{eq3:deriv0} from \eqref{expand1-S}. Note that 
   $\|M_0^{\top}M_0\|=\|M_0\|^2$ holds for the spectral norm.
   Dividing by $t^2$ we obtain for $t>0$
\begin{equation} \label{eq3:quadsin}
   \begin{aligned}
     \Big| \frac{\sin^2(\phi_s(t))}{t^2} - \|M_0\|^2 \Big| &
     = \Big| \frac{\|I_s - S(t)\|}{t^2}- \|M_0^{\top}M_0\| \Big| \\
     & \le \Big\| \frac{1}{t^2}(I_s - S(t))- M_0^{\top}M_0 \Big\|
     = o(1).
   \end{aligned}
   \end{equation}
   In case $\|M_0\|=0$ this shows $\lim_{t \searrow 0}t^{-1} \sin(\phi_s(t))=0$
   by the positivity of $\sin(\phi_s(t))$ and the continuity of
   the square root. Otherwise, we  have
   \begin{align*}
     |t^{-1} \sin(\phi_s(t)) - \|M_0\| |& = \frac{|t^{-2}\sin^2(\phi_s(t))
       -\|M_0\|^2|}{t^{-1} \sin(\phi_s(t)) + \|M_0\| }
      \le \frac{|t^{-2}\sin^2(\phi_s(t))
       -\|M_0\|^2|}{\|M_0\| } = o(1).
   \end{align*}
   Thus we have established that the derivative of $\sin(\phi_s(\cdot))$
   from the right at $t=0$ exists
   and is given by $\|M_0\|$. Since $\arcsin$ is continuously differentiable
   at $0$ and $\arcsin'(0)=1$ the same result holds for $\phi_s(\cdot)$.
   So far, formulas \eqref{eq3:deriv0} and \eqref{eq3:derivformula} have been shown for $W \in C^2$. Since
   \eqref{eq3:derivformula} does not involve the second derivative, an
   approximation argument finishes the proof. Approximate $W\in C^1$
   uniformly in $C^1$ by functions $W_n \in C^2$. Then \eqref{eq3:derivformula}
   holds for $W_n$ and one can pass to the limit as $n \to \infty$.

   {\it Proof of \eqref{expand1-S}:}
   Our first expansion follows from \eqref{expandW}
   \begin{align*}
     W(t)^{\top}W(t)&= W_0^{\top}W_0 + t\left[ \dot{W}_0^{\top}W_0+ W_0^{\top} \dot{W}_0 \right]\\
     &\phantom{=\ }+ \tfrac{t^2}{2}\left[ \ddot{W}_0^{\top}W_0 + W_0^{\top} \ddot{W}_0
       +2 \dot{W}_0^{\top} \dot{W}_0 \right] + o(t^2).
   \end{align*}
   Then we use the expansion of the geometric series
  \begin{align*} (I_s+tA_0+t^2 A_1 + o(t^2))^{-1}= I_s-t A_0 + t^2(A_0^2-A_1) + o(t^2)
  \end{align*}
  and obtain
  \begin{align*}
    (W(t)^{\top}W(t))^{-1} & = I_s - t \left[\dot{W}_0^{\top}W_0+ W_0^{\top} \dot{W}_0 \right] + o(t^2) \\
    &\phantom{=\ } + t^2\left[(\dot{W}_0^{\top}W_0+ W_0^{\top} \dot{W}_0)^2 -
      \tfrac{1}{2}(\ddot{W}_0^{\top}W_0 + W_0^{\top} \ddot{W}_0) -
      \dot{W}_0^{\top} \dot{W}_0 \right].
  \end{align*}
  With this we expand $I_s-S(t)$ as follows
  \begin{align*}
  I_s-S(t) & = I_s - \left\{ I_s + t W_0^{\top}\dot{W}_0 + \tfrac{t^2}{2}
     W_0^{\top}\ddot{W}_0 + o(t^2) \right\} \\
  &\phantom{=\ }\cdot \left\{ I_s - t \left[\dot{W}_0^{\top}W_0+ W_0^{\top} \dot{W}_0 \right]  \right. \\
    &\phantom{=\ }+ \left.  t^2\left[(\dot{W}_0^{\top}W_0+ W_0^{\top} \dot{W}_0)^2 -
      \tfrac{1}{2}(\ddot{W}_0^{\top}W_0 + W_0^{\top} \ddot{W}_0) -
      \dot{W}_0^{\top} \dot{W}_0\right] +o(t^2)  \right\}\\
     &\phantom{=\ }\cdot \left\{ I_s+t \dot{W}_0^{\top}W_0+ \tfrac{t^2}{2}
       \ddot{W}_0^{\top}W_0 + 
     o(t^2) \right\}.
  \end{align*}
  Collecting terms we see that the absolute and the linear term vanish
  so that we end up with a leading quadratic term.
  All second order derivatives drop out from this coefficient and so do $8$ of the
  remaining $10$ terms:
   \begin{align*}
    I_s-S(t) &=
    t^2 \left\{ \dot{W}_0^{\top}\dot{W}_0 - \dot{W}_0^{\top}W_0 W_0^{\top}\dot{W}_0\right\} + o(t^2) \\
    & = t^2 \dot{W}_0^{\top}(I_d - W_0W_0^{\top})\dot{W}_0+ o(t^2)=
    t^2 M_0^{\top}M_0 + o(t^2),
   \end{align*}
   where we used  that the projector $I_d - W_0W_0^{\top}$ is orthogonal.
   \end{proof}

\begin{remark}
  The crucial estimate \eqref{eq3:quadsin} shows that the left derivative
  at $\tau$ also exists but has a negative sign:
  \begin{equation*} \label{eq3:derivleft}
   \frac{d}{d t}\ang(V(\tau),V(t))_{ |t= \tau-} =-
   \|(I_d-P_{V(\tau)})\frac{dW}{dt}(\tau)  (W(\tau)^{\top}W(\tau))^{-1/2}\|.
  \end{equation*}
  The reason is that  principal angles always lie in $[0, \frac{\pi}{2}]$, so that $\ang(V(\tau),V(t))>0$ holds for $t <\tau$ as well. However, the
  jump in the derivative will not produce any difficulties in the following
  application to differential equations.
\end{remark}

Let us apply Theorem \ref{lemderivang} to subspaces generated by the evolution
of a linear nonautonomous system
\begin{equation} \label{eq3:cont1}
  \dot{u}(t) = A(t) u(t),\quad t \in [0,T), \quad A\in C([0,T),\R^{d,d}).
\end{equation}
Then the formula \eqref{eq3:derivformula} achieves a nice symmetric
and basis-free form.
\begin{corollary} \label{cor3:angformula}
  Let $\Phi(t,\tau)$, $t,\tau\in [0,T)$ denote the solution operator
  of \eqref{eq3:cont1} and consider subspaces generated by
  $V(t) = \Phi(t,0)V_0$ for some subspace $V_0 \in \cG(s,d)$.
  Then the following formula holds for $\tau \in [0,T)$
    \begin{equation}\label{eq3:angdiffeq}
      \frac{d}{dt} \ang(V(\tau),V(t))_{t=\tau+}= \|(I_d- P_{V(\tau)})A(\tau)
      P_{V(\tau)}\|,
    \end{equation}
    where $\|\cdot \|$ is the spectral norm in $\R^{d,d}$ and $P_{V(\tau)}$
    is the orthogonal projection onto $V(\tau)$.
\end{corollary}
\begin{proof}
  Let $V(0)=\range(W(0))$ for some $W(0)\in \R^{d,s}$ and note
  that $V(t) = \range(W(t))$  holds for $W(t)=\Phi(t,0)W(0)$, $t \in [0,T)$.
    By formula \eqref{eq3:derivformula} we obtain
    \begin{equation*} \label{eq3:deriveq}
      \begin{aligned}
   \frac{d}{d t}\ang(V(\tau),V(t))_{ |t= \tau+}& =
   \|(I_d-P_{V(\tau)})A(\tau) W(\tau)  (W(\tau)^{\top}W(\tau))^{-1/2}\|\\
   &= \|(I_d- P_{V(\tau)})A(\tau) P_{V(\tau)}\|.
   \end{aligned}
    \end{equation*}
    The last equality holds since multiplication by
    $(W(\tau)^{\top}W(\tau))^{-1/2}W(\tau)^{\top}\in \R^{s,d}$ preserves
    the spectral norm.
\end{proof}
In case  $s=1$  we have $V(\tau)=\Span(v(\tau))$, 
$P_{V(\tau)}= \frac{1}{\|v(\tau)\|^2} v(\tau) v(\tau)^{\top}$ where $v(\cdot)$ solves \eqref{eq3:cont1}.
The factor $\frac{1}{\|v(\tau)\|} v(\tau)^{\top}\in \R^{1,d}$ preserves the spectral norm, so that 
\eqref{eq3:angdiffeq} reads
\begin{equation*}
  \begin{aligned}
    \frac{d}{dt}\ang(v(\tau),v(t))_{t=\tau+} &=
    \frac{1}{\|v(\tau)\|}\big\|\big(I_d - \frac{v(\tau)v(\tau)^{\top}}{\|v(\tau)\|^2}\big)A(\tau)v(\tau)\big\|\\ & = \big\|\big(A(\tau)- \frac{v(\tau)^{\top}A(\tau)v(\tau)}{\|v(\tau)\|^2}I_d\big) \frac{v(\tau)}{\|v(\tau)\|} \big\|.
    \end{aligned}
  \end{equation*}
   In $\R^2$ we introduce the
  vector $v_{\perp}(\tau)=(-v_2(\tau),v_1(\tau))^{\top}$ orthogonal to
  $v(\tau)$, and the formula simplifies as follows: 
  \begin{equation} \label{eq4:2dform}
    \frac{d}{dt}\ang(v(\tau),v(t))_{t=\tau+}= \frac{| v_{\perp}(\tau)^{\top}A(\tau)v(\tau)|}{\|v(\tau)\|^2}.
  \end{equation}
  In Remark \ref{rem4:arnold} below we compare this expression in more
  detail with the terms used for the theory of  rotation numbers in \cite[Ch.6.5]{A1998}. 
  Here we discuss this expression  
  for a specific two-dimensional example the flow of which relates
  to a crucial example in \cite[Section 6.1]{BeFrHu20}.
  \begin{example}\label{ex3:1}
    For $ 0<\rho \le 1, \omega >0.$ consider
    \begin{equation} \label{eq3:model2D}
      A = \begin{pmatrix} 0 & - \rho^{-1} \omega \\ \rho \omega & 0 \end{pmatrix}, \quad \Phi(t,0) = \begin{pmatrix} \cos(t \omega) & -\rho^{-1}\sin(t\omega)\\
        \rho \sin(t \omega) & \cos(t \omega) \end{pmatrix}.
    \end{equation}
    Introducing the  matrices
    \begin{align} \label{eq3:introDT}
      D_{\rho}=\begin{pmatrix} 1 & 0 \\ 0 & \rho \end{pmatrix},
      \quad T_{\omega}= \begin{pmatrix} \cos(\omega) & - \sin(\omega) \\
        \sin(\omega) & \cos(\omega) \end{pmatrix}, \quad
      J= \begin{pmatrix} 0 & -1 \\ 1 & 0 \end{pmatrix},
    \end{align}
    we have  $A= \omega D_{\rho}JD_{\rho}^{-1}$, $v(\tau) = D_{\rho}T_{\tau \omega}D_{\rho}^{-1}v_0$, and $ I_2 - \|v\|^{-2}v v^{\top} = \|v\|^{-2} v_{\perp} v_{\perp}^{\top}$,
    where $v_{\perp}=Jv$ is  orthogonal to $v$.
    Then     a computation shows
    \begin{equation} \label{eq3:angformula}
    \begin{aligned}
      \alpha(\tau,v_0)&:=\frac{d}{dt}\ang(v(\tau),v(t))_{t= \tau+} = \frac{1}{\|v(\tau)\|^2}
      |v_{\perp}(\tau)^{\top} A v(\tau) |\\
      &\phantom{:}= \frac{\omega |v_0^{\top}D_{\rho}^{-1}T_{-\tau \omega} (D_{\rho}J)^2
      D_{\rho}^{-1}D_{\rho}T_{\tau \omega} D_{\rho}^{-1}v_0|}{\|D_{\rho}T_{\tau \omega}
    D_{\rho}^{-1}v_0\|^2}
       = \frac{\rho \omega \|D_{\rho}^{-1}v_0\|^2}{\|D_{\rho}T_{\tau \omega}
    D_{\rho}^{-1}v_0\|^2}.
    \end{aligned}
    \end{equation}
        The function $\alpha(\cdot,v_0)$ has period $\frac{\pi}{\omega}$ and since $\|D_{\rho}^{-1}v(\tau)\|= \|D_{\rho}^{-1}v_0\|$, it measures the angular speed of
    lines generated by points moving on an ellipse with semiaxes $1$ and $\rho$;
    see Figure \ref{ODE_ang_2D} for an illustration.
\begin{figure}[hbt] 
 \begin{center}
   \includegraphics[width=0.6\textwidth]{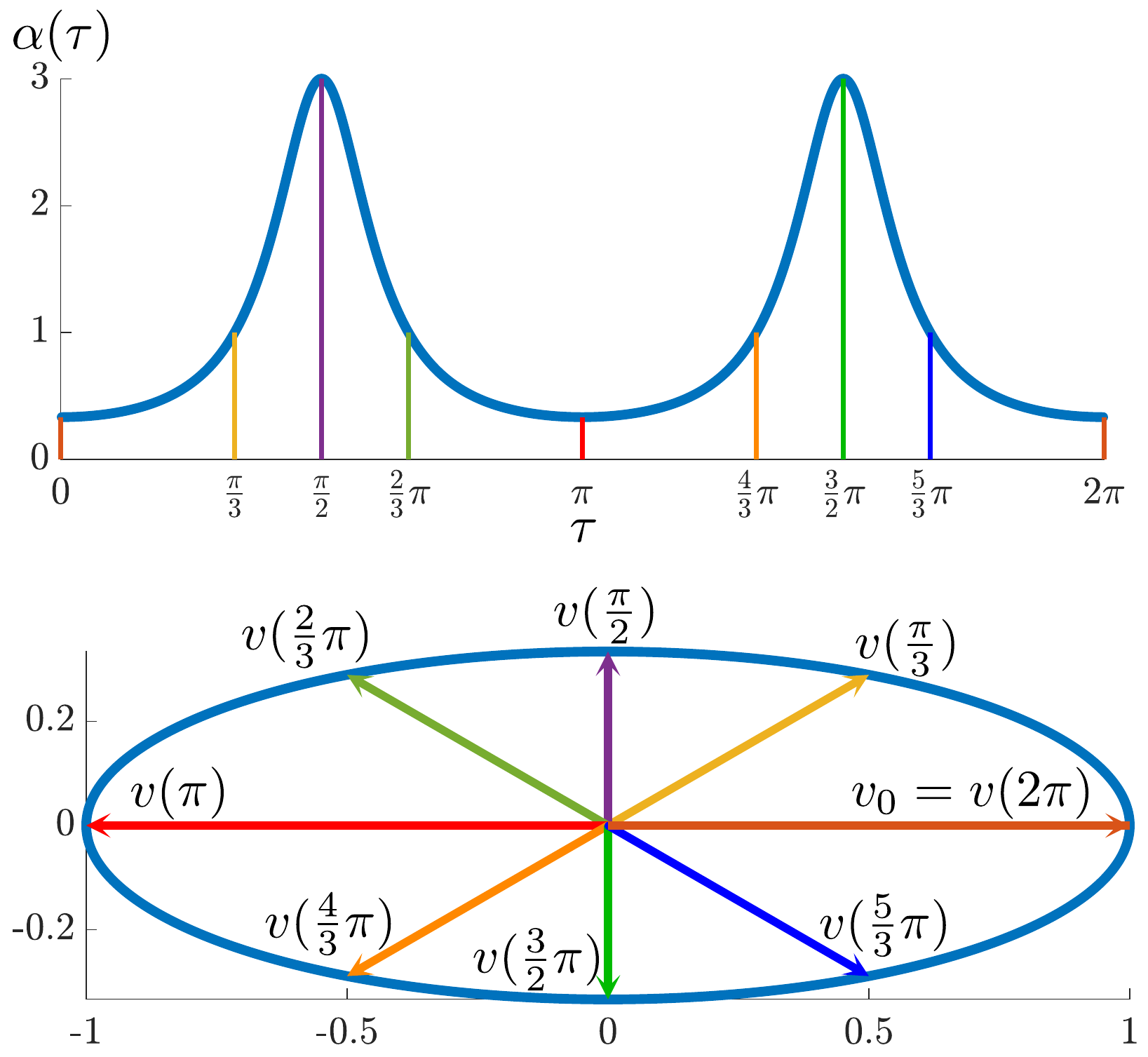}
 \end{center}
 \caption{\label{ODE_ang_2D} Angular speed $\alpha(\tau,v_0)$ vs. time $\tau$ for
   lines moving with the flow \eqref{eq3:model2D} for 
   $\omega=1$, $\rho=\frac{1}{3}$, $v_0
   = \begin{pmatrix}1&0\end{pmatrix}^\top$ (top).
      Lines generated by $v(\tau)$ moving on the ellipse
   $x^2 + \rho^{-2}y^2=1$ (bottom).}
\end{figure}

\end{example}


  \section{Application to angular values of linear dynamical systems}
  \label{sec4}
  In this section we take up the applications which motivated most of the
  results in previous sections: the theory of angular values for
  linear nonautonomous dynamical systems. A major step is to extend
  the definition in \cite{BeFrHu20} from discrete to continuous
  time systems via the  formulas from Section \ref{sec3.2}. Then we
  investigate for both cases the invariance of angular values under
  kinematic transformations. This will be the key to the explicit
  expressions derived in Section \ref{sec4.3} for the autonomous case
  with the help of Birkhoff's ergodic theorem.
  \subsection{Definition and invariance in the discrete case}
  \label{sec4.1}
From \cite[Section 3]{BeFrHu20} we recall several notions of
angular values for a nonautonomous linear
difference equation 
\begin{equation}\label{diffeq}
  u_{n+1} = A_n u_n,\quad A_n \in \R^{d,d},\quad n \in \N_0.
\end{equation}
We assume  that all matrices $A_n$ are invertible and that  
$A_n$ as well as $A_n^{-1}$ are uniformly bounded. The solution
operator is defined by
\begin{equation} \label{def4:solop}
  \Phi(n,m) = \begin{cases} A_{n-1}\cdot \ldots \cdot A_m, & \text{for} \;
    n >m , \\
    I_d, &\text{for} \; n=m , \\
    A_n^{-1} \cdot \ldots \cdot A_{m-1}^{-1}, & \text{for} \; n<m.
    \end{cases}
\end{equation}

\begin{definition} \label{defangularvalues}
  For $s\in \{1,\ldots,d\}$ and the operator $\Phi$ from \eqref{def4:solop}
  define the quantities 
  \begin{equation} \label{defsums}
    a_{m,n}(V) = \sum_{j=m}^{n} \ang(\Phi(j-1,0)V,\Phi(j,0)V) \quad 
   m,n \in \N,\; V\in \mathcal{G}(s,d).
    \end{equation}
    Then the  outer angular values of dimension $s$ are defined by
    \begin{equation*}\label{doutersup}
\begin{aligned}
  \theta_{s}^{\sup,\varlimsup } = \sup_{V \in \mathcal{G}(s,d)}
  \varlimsup_{n\to\infty} \frac{1}{n}  a_{1,n}(V), \quad
   \theta_s^{\sup ,\varliminf } =\sup_{V \in \mathcal{G}(s,d)}
  \varliminf_{n\to\infty} \frac{1}{n}  a_{1,n}(V)
\end{aligned}
    \end{equation*}
and the inner angular values of dimension $s$ are defined by
\begin{equation*}\label{douterinf}
\begin{aligned}
  \theta_{s}^{\varlimsup,\sup} =\varlimsup_{n\to\infty}
  \sup_{V \in \mathcal{G}(s,d)}
   \frac{1}{n}  a_{1,n}(V), \quad
   \theta_s^{\varliminf , \sup} = \varliminf_{n\to\infty}
   \sup_{V \in \mathcal{G}(s,d)} \frac{1}{n}  a_{1,n}(V).
\end{aligned}
\end{equation*}
\end{definition}

\begin{remark} \label{rem4:comparerot}
A common feature of these definitions is that one searches for subspaces
$V \in \cG(s,d)$ which maximize the longtime average of angles.
The attributes 'inner' and 'outer' then distinguish cases
where the supremum is taken inside or outside the limit as time
goes to infinity. At first glance, Definition \ref{defangularvalues} is reminiscent of the well studied notion of rotation numbers for homeomorphisms
of the circle; see \cite{dmVS1993,KH95, N1971}. Indeed, the notions agree in a very
specific two-dimensional  case (see \cite[Remark 5.3]{BeFrHu20}). However, in general there is
a fundamental difference:
angular values measure principal angles between subspaces rather than oriented
angles between vectors. There is even a difference in two dimensions,
since the principal angle of two lines can differ from the oriented angle of two vectors spanning
the one-dimensional subspaces.
Summarizing, angular values do not contain information about
orientation but allow to study the rotation of subspaces of arbitrary dimension
in systems of arbitrary dimension.
 For more comparisons with the literature
 we refer to \cite[Sections 1,5]{BeFrHu20}, and a discussion of the continuous time case
 can be found in Remark \ref{rem4:arnold} below.
\end{remark}

Note that the following relations hold 
\begin{equation}\label{outercon}
\begin{matrix}
\theta_s^{\sup ,\varliminf } & \le &  \theta_s^{\sup,\varlimsup }\\
\rle && \rle\\
\theta_s^{\varliminf, \sup} & \le &  \theta_s^{\varlimsup, \sup}
\end{matrix}  
\end{equation}
and that all inequalities can be strict; see \cite[Section 3]{BeFrHu20}.

In the following we generalize the invariance of principal angles from
Section \ref{sec2.3} to angular values. We consider a system which is kinematically similar to \eqref{diffeq}, i.e.\ we transform variables by
 $\tilde{u}_n=Q_n u_n$ with $Q_n \in \mathrm{GL}(\R^d)$ to obtain
\begin{equation} \label{difftransform}
  \tilde{u}_{n+1} = \tilde{A}_n \tilde{u}_n, \quad \tilde{A}_n =
  Q_{n+1}A_n Q_n^{-1}.
\end{equation}
The corresponding solution operators
$\tilde{\Phi}(n,m)$ and $\Phi(n,m)$ are related by
\begin{equation}\label{relatePhi}
  \tilde{\Phi}(n,m)Q_m= Q_n \Phi(n,m), \quad n \ge m.
\end{equation}
The following result is more general than \cite[Proposition 3.8]{BeFrHu20}
and was indicated there without proof.
\begin{proposition}(Invariance of angular values for discrete systems) \label{prop1.1}
  \begin{itemize}
    \item[(i)] Let $Q_n= q_n I_d$ with $q_n\neq 0$. Then each of the $4$ angular values of the systems \eqref{diffeq} and \eqref{difftransform} coincide. 
    \item[(ii)]
   Assume that the transformation matrices satisfy $\lim_{n \to \infty}Q_n=Q$,
  with  $Q \in \R^{d,d}$ being orthogonal.
  Then the same assertion as in (i) holds.
\item[(iii)]
  Assume that the transformation matrices satisfy
  $\lim_{n \to \infty}Q_n=Q$,
  with  $Q \in \R^{d,d}$ being invertible. Then an angular value vanishes
  for the system \eqref{diffeq} if and only if the same angular value vanishes for
  the system \eqref{difftransform}.
  \end{itemize}
  \end{proposition}
\begin{proof} {\it (i):}
  From \eqref{relatePhi} we have $\tilde{\Phi}(n,0)= \frac{q_n}{q_0} \Phi(n,0)$, hence the spaces $\tilde{\Phi}(n,0)V$ and $\Phi(n,0)V$ agree and the assertion follows from Definition
  \ref{defangularvalues}.

  {\it (ii):}
  The main step is to show that for any $\varepsilon>0$ there exists $N=N(\varepsilon)$ such that for all $n \ge N$ and for all $V \in \mathcal{G}(s,d)$
  \begin{equation} \label{basicPhi}
   S_n:= \frac{1}{n} \sum_{j=1}^n |\ang(\tilde{\Phi}(j-1,0)Q_0V,\tilde{\Phi}(j,0)Q_0V) - \ang(\Phi(j-1,0)V,\Phi(j,0)V)| \le \varepsilon.
  \end{equation}
  We use \eqref{relatePhi}, the triangular inequality for  $\ang(\cdot,\cdot)$
  from Proposition \ref{prop2g} and its invariance w.r.t.\
  orthogonal transformations from Section \ref{sec2.3} to obtain
  \begin{align*}
   S_n & = \frac{1}{n} \sum_{j=1}^n |\ang(Q_{j-1}\Phi(j-1,0)V,Q_j\Phi(j,0)V) - \ang(Q\Phi(j-1,0)V,Q \Phi(j,0)V)|\\
     & \le \frac{1}{n} \sum_{j=1}^n \ang(Q_{j-1}\Phi(j-1,0)V,Q\Phi(j-1,0)V) +\ang(Q_j\Phi(j,0)V,Q\Phi(j,0)V)\\
    & = \frac{1}{n} \sum_{j=1}^n \ang(Q^{\top}Q_{j-1}\Phi(j-1,0)V,\Phi(j-1,0)V) +\ang(Q^{\top}Q_j\Phi(j,0)V,\Phi(j,0)V)\\
    & \le \frac{C}{n} \sum_{j=1}^n \left( \|Q^{\top}Q_{j-1}-I_d\|+
    \|Q^{\top}Q_j-I_d\| \right),
  \end{align*}
  where $C$ is given by Lemma \ref{lem3:3}. Now choose $N_1=N_1(\varepsilon)$ and then $N=N(\varepsilon)$ such that
  \[
    \|Q^{\top}Q_j-I_d\| \le \frac{\varepsilon}{4C} \quad \forall j\ge N_1,
     \quad N \ge N_1\max(1,\tfrac{\pi }{\varepsilon}).
  \]
  Then we find for $n \ge N$
  \begin{align*}
    S_n & = \frac{1}{n} \left[ \sum_{j=1}^{N_1} \ldots \; + \sum_{j=N_1+1}^n
      \ldots \right]
      \le \frac{N_1}{n} \frac{\pi}{2} + \frac{n-N_1}{n} \frac{\varepsilon}{2}
      \le \varepsilon.
  \end{align*}
  Next consider the outer angular values $\theta_s^{\sup,\varlimsup}(A)$ and
  $\theta_s^{\sup,\varlimsup}(\tilde{A})$. Given $\varepsilon>0$ there exists $\tilde{N}(\varepsilon)$ and $V \in \mathcal{G}(s,d)$  such that for all $m \ge \tilde{N}(\varepsilon)$
  \begin{align*}
    &  \sup_{n\ge m} \frac{1}{n}\sum_{j=1}^n
    \ang(\Phi(j-1,0)V, \Phi(j,0)V) \ge \theta_s^{\sup,\varlimsup}(A)- \frac{\varepsilon}{2}.
  \end{align*}
  Using \eqref{basicPhi} with $\frac{\varepsilon}{2}$ we conclude
  \begin{align*}
     \sup_{n\ge m} \frac{1}{n}\sum_{j=1}^n
    \ang(\tilde{\Phi}(j-1,0)Q_0V, \tilde{\Phi}(j,0)Q_0V)  \ge \theta_s^{\sup,\varlimsup}(A)- \varepsilon.
  \end{align*}
  Taking the limit $m \to \infty$ and the supremum over $V\in\mathcal{G}(s,d)$  we obtain the estimate $\theta_s^{\sup,\varlimsup}(\tilde{A})\ge
  \theta_s^{\sup,\varlimsup}(A)- \varepsilon$. By symmetry, the same estimate holds
  when commuting $\tilde{A}$ and $A$. This proves
  $\theta_s^{\sup,\varlimsup}(\tilde{A})=\theta_s^{\sup,\varlimsup}(A)$.
  
  The other $3$ angular
  values are handled in a similar way by employing \eqref{basicPhi}.

  {\it (iii):}
  We show that there exists a constant $C_{\star}>0$ and
  for every $\varepsilon>0$ some $N=N(\varepsilon)$ such that for all $n \ge N$ and for all $V\in \mathcal{G}(s,d)$ the following holds
  \begin{equation} \label{boundPhi}
     \frac{1}{n} \sum_{j=1}^n \ang(\tilde{\Phi}(j-1,0)Q_0V,\tilde{\Phi}(j,0)Q_0V)
       \le \varepsilon + \frac{C_{\star}}{n} \sum_{j=1}^n \ang(\Phi(j-1,0)V,\Phi(j,0)V).
        \end{equation}
  Since  $Q_n,n \in \N$ and $Q$ are invertible the condition
  numbers of $Q_n$ are bounded
  \begin{align*}
    \|Q_n\| \|Q_n^{-1}\| \le \kappa, \quad \forall \, n \in \N.
  \end{align*}
  From \eqref{relatePhi} and the triangle inequality we obtain
  \begin{equation*} \label{estAsum}
    \begin{aligned}
      \frac{1}{n}& \sum_{j=1}^n \ang(\tilde{\Phi}(j-1,0)Q_0V,\tilde{\Phi}(j,0)Q_0V) 
        = \frac{1}{n} \sum_{j=1}^n \ang(Q_{j-1}\Phi(j-1,0)V,Q_j\Phi(j,0)V)\\
        & \le \frac{1}{n} \sum_{j=1}^n \ang(Q_{j-1}\Phi(j-1,0)V,Q_{j-1}\Phi(j,0)V)+ \frac{1}{n} \sum_{j=1}^n \ang(Q_{j-1}\Phi(j,0)V,Q_j\Phi(j,0)V)\\
        & \le \frac{C_{\star}}{n} \sum_{j=1}^n\ang(\Phi(j-1,0)V,\Phi(j,0)V)
        +\frac{1}{n} \sum_{j=1}^n \ang(Q_{j-1}Q_j^{-1}Q_j\Phi(j,0)V,Q_j\Phi(j,0)V)\\
        & \le \frac{C_{\star}}{n} \sum_{j=1}^n\ang(\Phi(j-1,0)V,\Phi(j,0)V)
        + \frac{C}{n} \sum_{j=1}^n \|Q_{j-1}Q_j^{-1}-I_d \|,
    \end{aligned}
  \end{equation*}
  with constants $C_{\star}=\pi \kappa(1+\kappa)$ from Lemma \ref{lem3:1}
  and $C$ from Lemma \ref{lem3:3}. Let $K$ satisfy $\|Q_{j-1}Q_j^{-1}-I_d\|\le K$
  for all $j \ge 1$. Now choose
  $N_1=N_1(\varepsilon)$ and then $N=N(\varepsilon)$ such that
  \[
    \|Q_{j-1}Q_j^{-1}-I_d\| \le \frac{\varepsilon}{2C} \quad \forall j\ge N_1,
     \quad 
    N \ge N_1\max(1,\tfrac{2CK }{\varepsilon}).
  \]
  Then we find for $n \ge N$ 
  \begin{align*}
  \frac{C}{n} \sum_{j=1}^n \|Q_{j-1}Q_j^{-1}-I_d \|   & = \frac{1}{n} \left[ \sum_{j=1}^{N_1} \ldots \; + \sum_{j=N_1+1}^n
      \ldots \right]
      \le \frac{N_1}{n} \frac{CK}{2} + \frac{n-N_1}{n} \frac{\varepsilon}{2}
      \le \varepsilon,
  \end{align*}
  which proves \eqref{boundPhi}. Similar to the above, taking the $\varlimsup_{n \to \infty}$ in \eqref{boundPhi} and then the supremum over $V\in \mathcal{G}(s,d)$ yields
  \begin{align*}
    \theta_s^{\sup,\varlimsup}(\tilde{A}) \le \varepsilon + C_{\star}\theta_s^{\sup,\varlimsup}(A).
  \end{align*}
  By symmetry, the same estimate holds when commuting $\tilde{A}$ and $A$.
  This shows that both angular values can only vanish simultaneously.
  The remaining angular values are analyzed with the help of \eqref{boundPhi} in  a similar
  way.
\end{proof}

\subsection{Definition and invariance in the continuous case}
\label{sec4.2}
 In this section we define and investigate angular values for
 the continuous time system \eqref{eq3:cont1}. This solves an open problem
 discussed in \cite[Outlook]{BeFrHu20}.
 Let us first consider a finite interval $[0,T]$ and
 take Definition \ref{defsums} in the discrete case as motivation for
 a proper definition in continuous time.
 We choose a mesh on $[0,T]$ with stepsize $h = \frac{T}{N}$ and
 look at the average angle between successive images $V_{jh}=\Phi(jh,0)V$ of $V \in \cG(s,d)$ under
 the time $h$-flow of \eqref{eq3:cont1}
 \begin{align*}
   \theta_s(h,T,V) & := \frac{1}{T} \sum_{j=1}^{N} \ang(V_{(j-1)h},
   V_{jh})\\
   & \phantom{:}=\frac{h}{T} \sum_{j=1}^{N}\frac{1}{h}\left[ \ang(V_{(j-1)h},
     \Phi(jh,(j-1)h)V_{(j-1)h})- \ang(V_{(j-1)h},V_{(j-1)h}) \right].
 \end{align*}
 From Corollary \ref{cor3:angformula} we know that the difference quotients
 converge as $h \to 0$ and assuming the same for the integral,
  we obtain
 \begin{align*} \label{eq3:lim1}
   \lim_{h \to 0} \theta_s(h,T,V) & =\frac{1}{T} \int_0^T \frac{d}{dt}
   \ang(\Phi(\tau,0)V,\Phi(t,\tau)
   \Phi(\tau,0) V)_{| t =\tau+} d\tau.
 \end{align*}
 Taking the limit  $ T \to \infty$ and using the formula \eqref{eq3:angdiffeq}
 for the derivative, we end up with the following analog of
 Definition \ref{defangularvalues}.
 \begin{definition} \label{contangularvalues}
   Let the system \eqref{eq3:cont1} be given on $[0,\infty)$ and
     let $\Phi$ be its solution operator. For $V\in \cG(s,d)$ with $s \in \{1,\ldots,d\}$
 define the quantities 
  \begin{equation} \label{contsums}
    a_{t,T}(V) = \int_{t}^T  \| (I_d - 
   P_{\Phi(\tau,0)V})A(\tau)
   P_{\Phi(\tau,0) V}\| d\tau \quad 
   t,T \in [0,\infty),
  \end{equation}
  where $ P_{\Phi(\tau,0)V}$ denotes the orthogonal projection onto $\Phi(\tau,0)V$.
    Then the  outer angular values of dimension $s$ are defined by
    \begin{equation*}
\begin{aligned}
  \vartheta_{s}^{\sup,\varlimsup } = \sup_{V \in \mathcal{G}(s,d)}
  \varlimsup_{T\to\infty} \frac{1}{T}  a_{0,T}(V), \quad
   \vartheta_s^{\sup ,\varliminf } =\sup_{V \in \mathcal{G}(s,d)}
  \varliminf_{T\to\infty} \frac{1}{T}  a_{0,T}(V),
\end{aligned}
    \end{equation*}
and the inner angular values of dimension $s$ are defined by
\begin{equation*}
\begin{aligned}
  \vartheta_{s}^{\varlimsup,\sup} =\varlimsup_{T\to\infty}
  \sup_{V \in \mathcal{G}(s,d)}
   \frac{1}{T}  a_{0,T}(V), \quad
   \vartheta_s^{\varliminf , \sup} = \varliminf_{T\to\infty}
   \sup_{V \in \mathcal{G}(s,d)} \frac{1}{T}  a_{0,T}(V).
\end{aligned}
\end{equation*}
 \end{definition}
 \begin{remark} \label{rem4:arnold} In \cite[Ch.6.5]{A1998} L.\ Arnold gives
   a general definition of
   rotation numbers for vectors in  linear continuous time systems.
   In dimension two his definition amounts to
   \begin{align} \label{eq4:arnold}
     \rho(A)= \lim_{T \to \infty}\frac{1}{T} \int_0^T \frac{v_{\perp}^{\top}(t) A(t)v(t)}{\|v(t)\|^2} dt,
     \quad
     v_{\perp}(t) = (-v_2(t), v_1(t))^{\top},
   \end{align}
   where  $v(t) = \Phi(t,0)v_0$ solves \eqref{eq3:cont1} and the orthogonal
   vector $v_{\perp}(t)$ is oriented counterclockwise.  The integrand in 
   \eqref{eq4:arnold} differs just by the absolute value from the
   expression \eqref{eq4:2dform} used for $\vartheta_s^{\sup,\varlimsup}$ in Definition \ref{defangularvalues} above. In this sense the maximum rotation number
   in \cite[Ch.6.5]{A1998} is an oriented version of our angular value for
   $s=1$.
   In higher space dimensions the approach in \cite[Ch.6.5]{A1998} differs
  significantly from angular values with $s=1$. While we use principal
  angles between one-dimensional subspaces, the generalization in
  \cite[Ch.6.5]{A1998} utilizes the flow induced by \eqref{eq3:cont1} on the Stiefel manifold $\mathrm{St}(2,d)$ and then transfers the two-dimensional rotation number to
  appropriately moving frames in this manifold.
         \end{remark}
As in \eqref{outercon} we have the following obvious relations
 \begin{equation*}\label{eq:outerrelate}
\begin{matrix}
\vartheta_s^{\sup ,\varliminf } & \le &  \vartheta_s^{\sup,\varlimsup }\\
\rle && \rle\\
\vartheta_s^{\varliminf, \sup} & \le &  \vartheta_s^{\varlimsup, \sup}
\end{matrix}  
 \end{equation*}
 and  the elementary estimate $\vartheta_s^{\varlimsup,\sup} \le \sup_{t\ge 0}\|A(t)\|$
  if $A(\cdot)$ is uniformly bounded.
   
 \begin{example} \label{ex4:1} (Example \ref{ex3:1} revisited)\\
   The angle function $\alpha(\tau,v_0)$ from \eqref{eq3:angformula} is $\frac{\pi}{\omega}$-periodic. Therefore, we obtain from Definition
   \ref{contangularvalues}
   \begin{align*}
\vartheta_{1}^{\sup,\varlimsup } & =
   \vartheta_1^{\sup ,\varliminf } =\sup_{v_0 \neq 0}
   \frac{\omega}{\pi}\int_0^{\frac{\pi}{\omega}} \alpha(\tau,v_0) d\tau
   = \frac{\rho \omega^2}{\pi} \sup_{v_0 \neq 0}\int_0^{\frac{\pi}{\omega}}
   \frac{\|D_{\rho}^{-1}v_0\|^2}{\|D_{\rho}T_{\tau \omega} D_{\rho}^{-1} v_0\|^2} d \tau \\
   &=\frac{\rho \omega^2}{\pi} \sup_{\|v_0\|=1}\int_0^{\frac{\pi}{\omega}}
   \frac{1}{\|D_{\rho}T_{\tau \omega} v_0\|^2} d \tau=
   \frac{\rho \omega}{\pi} \sup_{\|v_0\|=1}\int_0^{\pi}
   \frac{1}{\|D_{\rho}T_{t} v_0\|^2} dt.
   \end{align*}
   The integrand is $\pi$-periodic, and the integral stays the same
   if we replace $v_0$ by $T_{\varphi}v_0$. Therefore, the integral
   is independent of $v_0$ and by taking $v_0=\begin{pmatrix}1 & 0 \end{pmatrix}^{\top}$ we obtain
   \begin{align*}
     \vartheta_{1}^{\sup,\varlimsup } & =
     \vartheta_1^{\sup ,\varliminf }= \frac{\rho \omega}{\pi}\int_0^{\pi}
     \frac{1}{\cos^2(t) + \rho^2 \sin^2(t)} dt =\omega.
   \end{align*}
   This result is natural in view of the rotational flow \eqref{eq3:model2D}
   on an ellipse, but it is in stark contrast to the resonance phenomena
   observed (and proved) for the two-dimensional discrete case in \cite[Section 6]{BeFrHu20}.
In continuous time
we also have $ \vartheta_{1}^{\varlimsup,\sup }  =
\vartheta_1^{\varliminf,\sup }= \omega$ for the inner angular values.
This can be seen as follows. Write $T= n\frac{2 \pi}{\omega}+ \tau$
with $n \in \N_0$,  $0 \le \tau < \frac{2 \pi}{\omega}$ and obtain from
the periodicity
\begin{align*}
  \frac{1}{T}\int_0^T \alpha(t,v_0) dt& = \frac{\omega}{2 \pi n + \tau \omega}
  \Big(\int_0^{n \frac{2 \pi}{\omega}} \alpha(t,v_0) dt
  + \int_{n\frac{2 \pi}{\omega}}^T\alpha(t,v_0) dt\Big)\\
  & = \frac{\omega}{2 \pi n + \tau \omega} \Big( 2 \pi n + \int_0^{\tau}
  \alpha(t,v_0) dt \Big).
  \end{align*}
The last integral is uniformly bounded. Hence taking first the supremum
with respect to $v_0$ and then the limit $n \to \infty$ shows
\begin{align*}
  \lim_{T \to \infty}\sup_{v_0 \neq 0}\int_0^T\alpha(t,v_0) dt = \omega.
\end{align*}
   \end{example}

 As in Section \ref{sec4.1} we study the invariance of angular values
 for \eqref{eq3:cont1} on $[0,\infty)$ under a kinematic transformation, i.e.\ we
 transform variables by $v(t)=Q(t)u(t)$, $t \ge0$ where $Q \in C^1([0,\infty),
  \mathrm{GL}(\R^d))$:
 \begin{equation} \label{eq4:cont2}
     \dot{v}(t)= \tilde{A}(t)v(t), \quad \tilde{A}(t) = (\dot{Q}(t)+
     Q(t)A(t))Q(t)^{-1}, \quad t \ge 0.
 \end{equation}
 Note that the corresponding solution operators $\tilde{\Phi}(t,\tau)$ and
 $\Phi(t,\tau)$ are related by
 \begin{equation} \label{contrelatePhi}
     \tilde{\Phi}(t,\tau) Q(\tau) = Q(t) \Phi(t,\tau), \quad \tau,t\ge 0.
 \end{equation}
 \begin{proposition}(Invariance of angular values for continuous systems)
   \label{propcont}
        Let the matrices $A\in C([0,\infty),\R^{d,d})$ be  uniformly
       bounded and let $Q \in C^1([0,\infty),\mathrm{GL}(\R^d))$ be given. 
  \begin{itemize}
    \item[(i)] If $Q(t)=q(t)I_d$ with $q\in C^1([0,\infty),\R)$ 
then each of the $4$ angular values of the system \eqref{eq3:cont1} coincides with the corresponding angular value of the system
  \eqref{eq4:cont2}.
    \item[(ii)] The same assertion as in (i) holds
    if $\lim_{t\to \infty}\dot{Q}(t)=0$ and if $\lim_{t \to \infty}Q(t)=Q_{\infty}$ exists and is orthogonal.
\item[(iii)] If $\lim_{t\to \infty}\dot{Q}(t)=0$ and $Q(\cdot)$, $Q(\cdot)^{-1}$ are uniformly bounded then
   an angular value of
   the system \eqref{eq3:cont1} vanishes
   if and only if the corresponding angular value of
  the system \eqref{eq4:cont2} vanishes.
  \end{itemize}
 \end{proposition}
 
 \begin{proof}
   {\it (i):} In this case we have $\tilde{A}(t)=\frac{\dot{q}(t)}{q(t)}I_d+A(t)$
   and $\tilde{\Phi}(t,0)= \frac{q(t)}{q(0)}\Phi(t,0)$. Therefore, we
   have the equality $\Phi(t,0)V= \tilde{\Phi}(t,0)V$ for all $V\in \cG(s,d)$
   and the expressions
   $a_{t,T}(V)$ in \eqref{contsums} agree with those of the transformed system.
   This proves the invariance.
   
   {\it (ii):}
   Similar to \eqref{basicPhi} we show that for every $\varepsilon >0$ there exists $T_0 >0$ such that for all $T \ge T_0$ and for all $ V \in \cG(s,d)$:
   \begin{equation} \label{Phitildest}
     \begin{aligned}
       \frac{1}{T}\Big| \int_0^T \|(I_d-P_{\tilde{\Phi}(t,0)Q(0)V})\tilde{A}(t)
     P_{\tilde{\Phi}(t,0)Q(0)V} \| - \| (I_d-P_{\Phi(t,0)V}) A(t)
     P_{\Phi(t,0)V}  \| dt \Big| \le \varepsilon.
     \end{aligned}
   \end{equation}
   To see this, observe that \eqref{contrelatePhi} and Proposition \ref{prop2g}, Lemma \ref{lem3:3}
   lead to
   \begin{equation} \label{eq4:projest}
   \begin{aligned}
    & \| P_{\tilde{\Phi}(t,0)Q(0)V} -Q_{\infty}P_{\Phi(t,0)V}Q_{\infty}^{\top}\| 
     =\|P_{Q(t)\Phi(t,0)V}-P_{Q_{\infty}\Phi(t,0)V}\|  \\
     & \le \ang(Q(t)\Phi(t,0)V, Q_{\infty}\Phi(t,0)V)=
     \ang(Q_{\infty}^{\top}Q(t)\Phi(t,0)V, \Phi(t,0)V)\\
     & \le C \|Q_{\infty}Q(t)-I_d\| = C\|Q(t)-Q_{\infty}\| \to 0
     \quad \text{as} \; t \to \infty.
   \end{aligned}
   \end{equation}
   With \eqref{eq4:cont2} and the  orthogonality of $Q_{\infty}$ we
   estimate the lefthand side of \eqref{Phitildest} by
  \begin{align*} & 
      \frac{1}{T} \int_0^T\Big| \|(I_d-P_{\tilde{\Phi}(t,0)Q(0)V})\tilde{A}(t)
     P_{\tilde{\Phi}(t,0)Q(0)V} \| - \| Q_{\infty}(I_d-P_{\Phi(t,0)V}) A(t)
     P_{\Phi(t,0)V} Q_{\infty}^{\top} \|\Big| dt \\ 
   & \le \frac{1}{T} \Big[ \int_0^T \| (I_d-P_{\tilde{\Phi}(t,0)Q(0)V})\dot{Q}(t)
      Q^{-1}(t)P_{\tilde{\Phi}(t,0)Q(0)V} \|\\ 
        & \qquad \qquad + \|(I_d-P_{\tilde{\Phi}(t,0)Q(0)V})Q(t)A(t)Q(t)^{-1}
        P_{\tilde{\Phi}(t,0)Q(0)V}\\
        &\qquad \qquad -Q_{\infty}(I_d-P_{\Phi(t,0)V})Q_{\infty}^{\top}Q_{\infty}A(t)
        Q_{\infty}^{\top}Q_{\infty}P_{\Phi(t,0)V}
        Q_{\infty}^{\top} \| dt \Big].
  \end{align*}
  From the convergence of projectors in \eqref{eq4:projest} and from
  $Q(t)\to Q_{\infty}$, $Q(t)^{-1} \to Q_{\infty}^{\top}$,
  $\dot{Q}(t) \to 0$ as $t \to \infty$ we conclude via the boundedness of
  $A(t)$ that the integrand converges to zero as $t \to \infty$ uniformly in
  $V \in \cG(s,d)$. Thus \eqref{Phitildest} can be achieved by taking
  $T$ sufficiently large. 
  
  Consider now the outer angular values $\vartheta_s^{\sup,\varlimsup}(A)$ and
  $\vartheta_s^{\sup,\varlimsup}(\tilde{A})$. Given $\varepsilon>0$ there exists
  $V \in \mathcal{G}(s,d)$  and $T_0=T_0(\varepsilon,V)$ such that 
  \begin{align*}
    &  \sup_{T\ge T_0} \frac{1}{T}\int_0^T
    \|(I_d - P_{\Phi(t,0)V}A(t)P_{\Phi(t,0)V}\| dt  \ge \vartheta_s^{\sup,\varlimsup}(A)- \frac{\varepsilon}{2}.
  \end{align*}
  Using \eqref{Phitildest} with $\frac{\varepsilon}{2}$ we conclude for $T_0$
  sufficiently large
  \begin{align*}
    \sup_{T\ge T_0} \frac{1}{T}\int_0^T
    (I_d - P_{\tilde{\Phi}(t,0)Q(0)V})\tilde{A}(t)P_{\tilde{\Phi}(t,0)V}\| dt
     \ge \vartheta_s^{\sup,\varlimsup}(A)- \varepsilon.
  \end{align*}
  Taking the limit $T_0 \to \infty$ and the supremum over $V \in \cG(s,d)$  we obtain the estimate $\vartheta_s^{\sup,\varlimsup}(\tilde{A})\ge
  \vartheta_s^{\sup,\varlimsup}(A)- \varepsilon$. By symmetry, the same estimate holds
  when commuting $\tilde{A}$ and $A$. This proves
  $\vartheta_s^{\sup,\varlimsup}(\tilde{A})=\vartheta_s^{\sup,\varlimsup}(A)$.
    The other $3$ angular
  values are handled in a similar way by employing \eqref{Phitildest}.
  \medskip
  
  {\it (iii):}
  For this proof it suffices to show that there exists a constant
  $C_{\star}$ and for every $\varepsilon >0$ some  $T_0$ such that for all
  $T\ge T_0$ and $V \in \cG(s,d)$ the following holds
   \begin{equation} \label{boundcondPhi}
     \frac{1}{T} \int_0^T \|(I_d-P_{\tilde{\Phi}(t,0)Q(0)V}) \tilde{A}(t)
     P_{\tilde{\Phi}(t,0)Q(0)V}\| dt
       \le \varepsilon + \frac{C_{\star}}{T} \int_0^T
      \| (I_d-P_{\Phi(t,0)V}) A(t) P_{\Phi(t,0)V} \| dt.
   \end{equation}
   Instead of \eqref{eq4:projest} we use the following identities
   \begin{equation*} \label{eq4:projprod}
     \begin{aligned}
       (I_d-P_{Q(t)\Phi(t,0)V})Q(t)& = (I_d-P_{Q(t)\Phi(t,0)V})Q(t)(I_d - P_{\Phi(t,0)V}),\\
       Q(t)^{-1}P_{Q(t)\Phi(t,0)V}& = P_{\Phi(t,0)V}Q(t)^{-1}P_{Q(t)\Phi(t,0)V}.
     \end{aligned}
   \end{equation*}
   These lead to the following estimate
   \begin{align*}
    & \|(I_d-P_{\tilde{\Phi}(t,0)Q(0)V}) \tilde{A}(t)
     P_{\tilde{\Phi}(t,0)Q(0)V}\| = \|(I_d-P_{Q(t)\Phi(t,0)V}) \tilde{A}(t)
     P_{Q(t) \Phi(t,0)V}\|  \\
     & \le \|(I_d-P_{Q(t)\Phi(t,0)V}) \dot{Q}(t) Q(t)^{-1}
     P_{Q(t) \Phi(t,0)V} \| \\
     &\phantom{=\ } + \|(I_d-P_{Q(t)\Phi(t,0)V})Q(t)(I_d - P_{\Phi(t,0)V})A(t)
     P_{\Phi(t,0)V}Q(t)^{-1}P_{Q(t)\Phi(t,0)V}\|\\
     & \le \|Q^{-1}\|_{\infty} \|\dot{Q}(t)\| + \|Q\|_{\infty}\|Q^{-1}\|_{\infty}
     \|(I_d - P_{\Phi(t,0)V})A(t) P_{\Phi(t,0)V}\|.
   \end{align*}
   Our assumptions on $Q$  then guarantee that \eqref{boundcondPhi} holds for
   $T$ sufficiently large with the constant
   $C_{\star}=\|Q\|_{\infty} \|Q^{-1}\|_{\infty}$.
   Suppose that $\vartheta_s^{\sup,\varlimsup}(A)=0$. Then for all $V \in \cG(s,d)$ and for each $\varepsilon>0$ there exists $T_0=T_0(\varepsilon,V)$ such
   that 
   \begin{align*}
     \sup_{T\ge T_0} \frac{1}{T}\int_0^T \|(I_d - P_{\Phi(t,0)V})A(t)P_{\Phi(t,0)V}\|
     dt \le \varepsilon.
   \end{align*}
   Combining this with \eqref{boundcondPhi} we obtain
   \begin{align*}
     \sup_{T\ge T_0} \frac{1}{T} \int_0^T \|(I_d-P_{\tilde{\Phi}(t,0)Q(0)V}) \tilde{A}(t)
      P_{\tilde{\Phi}(t,0)Q(0)V}\| dt \le (1+C_{\star})\varepsilon
   \end{align*}
   for  $T_0$ sufficiently large. Therefore, the limit as $T_0 \to \infty$ vanishes
   and recalling that this holds for all $V$, we end up with
   $\vartheta_s^{\sup,\varlimsup}(\tilde{A})=0$.
   The other angular values are treated by using \eqref{boundcondPhi}
   in a similar manner.
 \end{proof}
 
 \begin{remark} If we replace $A(t)$ in \eqref{eq3:cont1} by $A(t)+\lambda(t)I_d$, then angular values remain unchanged. This follows from case (i) by setting
   $q(t) = \exp\big(\int_0^t \lambda(\tau) d\tau\big)$. In particular, we can choose
   $\lambda(t) = -\frac{1}{d} \mathrm{tr}(A(t))$. Hence the computation of angular values
   can be restricted to systems \eqref{eq3:cont1} satisfying
   $\mathrm{tr}(A(t))=0$ for $t\in [0,\infty)$.
 \end{remark}
 \subsection{Some formulas for autonomous systems}
 \label{sec4.3}
 In the following we apply the results of the previous section  to the autonomous system
 \begin{equation} \label{eq4:autonomous}
   \dot{u}= A u, \quad A \in \R^{d,d}.
 \end{equation}
 Proposition \ref{propcont} (ii) suggests to look at the real quasitriangular Schur form of $A$:
 
\begin{equation} \label{eq4:realschur}
  \begin{aligned}
   Q A Q^{\top} &= \begin{pmatrix} \Lambda_{11} & \Lambda_{12} & \cdots & \Lambda_{1k}\\
     0 & \Lambda_{22} & \cdots & \Lambda_{2k} \\
     \vdots & \ddots & \ddots & \vdots \\
     0 & \cdots & 0 & \Lambda_{kk}
   \end{pmatrix}, \quad Q^{\top}Q = I_d, \quad \Lambda_{ij} \in \R^{d_i,d_j}, \\
   \Lambda_{jj}&=
   \begin{cases} \begin{pmatrix} \beta_j & - \frac{\omega_j}{\rho_j}  \\  \rho_j \omega_j &\beta_j
   \end{pmatrix}\in \R^{2,2},& d_j=2, \omega_j> 0,\quad 0 < \rho_j \le 1, \\
      \lambda_j=\beta_j\in \R, & d_j=1,
   \end{cases}\\
   & \beta_1 \ge \cdots \ge \beta_k.
     \end{aligned}
  \end{equation}
We refer to \cite[Theorem 2.3.4]{HJ2013} and emphasize the special form\footnote{This orthogonal normal form of a $2\times2$ matrix with complex eigenvalues
is readily obtained but not mentioned in \cite[Theorem 2.3.4]{HJ2013}. The
parameter $\rho_j$ is the semiaxis of the ellipse $x^2+\rho_j^{-2}y^2=1$ 
invariant under the map $\Lambda_{jj}$. Its value  is important for our explicit computations in Section
\ref{sec4.3}.} of
the blocks $\Lambda_{jj}$ with nonreal eigenvalues $\lambda_j=\beta_j + i \omega_j$; cf.\ Example \ref{ex3:1}. Further, we introduce
   \begin{align*}
                 J_{\C}= \{ j \in \{1,\ldots,k\}:d_j=2 \}.
   \end{align*}
   As a preparation we determine the asymptotic behavior of the orthogonal
   projections
  $P_{e^{tA}V}$ for a given element $V=\range(W) \in \cG(s,d)$ with $W \in \R^{d,s}$.
By column operations we can put $W$ 
into block column echelon form. Within the blocks, the number of rows
is given by the sequence $d_1,\ldots,d_k$ from 
\eqref{eq4:realschur} while the number of columns is either $1$ or $2$:
            \begin{equation} \label{eq4:columnechelon}
            \begin{aligned}
              W&= \begin{pmatrix} 0 & 0 &\cdots & 0 \\
               W_{k_1+1,1} & 0 & \cdots & 0 \\
               \vdots  & W_{k_2+1,2} &  & \vdots\\
               \vdots& \vdots & \ddots & W_{k_r+1,r} \\
               \vdots & \vdots & \vdots & \vdots \\
               W_{k,1} & W_{k,2} & \cdots & W_{k,r}
              \end{pmatrix}, \quad
                              \big(W_{ij}\in \R^{d_i,b_j}\big)_{i=1,\ldots,k}^{j=1,\ldots,r}, \\ 
                    s&=  \displaystyle\sum_{j=1}^r b_j , \quad \text{where}\quad
            b_j  \begin{cases} =1, & \text{if}\, d_{k_j+1}=1,\\
                  \in \{1,2\}, & \text{if}\, d_{k_j+1}= 2 .
                \end{cases} 
                           \end{aligned}
            \end{equation}
         Here the indices 
         $0 \le k_1 < k_2  < \cdots < k_r<k$ are determined such that
         \begin{align*}  \label{eq4:blockechelon}
      W_{ij}= 0 \; (i=1,\ldots,k_j), \quad \rank(W_{k_{j}+1,j})=b_j, \quad j=1,\ldots,r.
      \end{align*}
         Note that the leading entry $W_{k_{j}+1,j}$ is of size $2 \times 1$
         or $2 \times 2$ if $k_j+1 \in J_{\C}$ and of size $1 \times 1$
         otherwise. Further, the numbers $k_j$ and the ranges $\range(W_{k_j+1,j})$, $j=1,\ldots,r$ are unique.
         
      \begin{example} \label{Ex4:2}
        Consider a continuous system \eqref{eq4:autonomous} in $\R^4$  with 
        matrices
        \begin{equation} \label{eq4:A4}
        A=  A(\omega,\rho) = \begin{pmatrix}
            \Lambda_{11} & 0 \\ 0 & \Lambda_{22} \end{pmatrix},
          \quad
          \Lambda_{jj}= \begin{pmatrix} j-1 & - \frac{\omega_j}{\rho_j}
            \\  \rho_j \omega_j & j-1 \end{pmatrix},\quad j=1,2, 
        \end{equation}
        where $\omega=(\omega_1,\omega_2) \in (0,\infty)^2$,  $\rho=(\rho_1,\rho_2)\in (0,1]^2$. We have $d=4$, $k=2$, $J_{\C}=\{1,2\}$ and we chose different 
   diagonal elements in $\Lambda_{jj}$, $j=1,2$ so that Lemma \ref{lem4:reduce}  below applies.
   Then the following choices for the column echelon form \eqref{eq4:columnechelon} are possible:
   \begin{equation} \label{eq4:indexchoices}
   \begin{aligned}
     s=1: & \quad (k_1=0,b_1=1) \;\text{or}\; (k_1=1,b_1=1), \\
     s=2: & \quad (k_1=0, b_1=1, k_2=1,b_2=1) \;\text{or}\;
     (k_1=0,b_1=2) \;\text{or}\; (k_1=1,b_1=2), \\
     s=3: &\quad (k_1=0,b_1=1,k_2=1,b_2=2) \;\text{or}\; (k_1=0,b_1=2,k_2=1,b_2=1).
   \end{aligned}
   \end{equation}
      As we will see below, the cases where at least one
   $b_j=1$ occurs, are most relevant for the angular value.
      \end{example}
      To take care of multiple real eigenvalues we determine
       numbers $\ell_j\in \{1,\ldots,k-1\}$ for $j=1,\ldots,r$
     with
      \begin{align} \label{eq4:lind}
        \beta_{k_j+1}= \cdots = \beta_{k_j+\ell_j} >
        \beta_{k_j+\ell_j+1}.
      \end{align}
      Further, we assume isolated real parts of the eigenvalues
       $\lambda_j=\beta_j +i \omega_j$, $j \in J_{\C}$, i.e.
   \begin{align} \label{eq4:ordIm}
   \beta_j\neq \beta_{\nu}
   \quad \forall  j\in J_{\C},\ \nu \in \{1,\ldots,k\},\ \nu \neq j,
   \end{align}
   which implies $\ell_j=1$ if $k_j+1 \in J_{\C}$.
   With these data we define
      $W^{\infty}\in \R^{d,s}$ and its image space $V^{\infty}=\range(W^{\infty})$ by
            \begin{align} \label{eq4:defWinfty}
        W^{\infty}_{i,j}&= \begin{cases} W_{i,j}, & k_j+1\le i \le k_j + \ell_j,
          \quad j\in \{1,\ldots,r\}, \\
                   0, & \text{otherwise}. 
        \end{cases}
      \end{align}
   Complex eigenvalues lead to single blocks in $W^{\infty}$ while
  a sequence of $\ell_j$ identical real eigenvalues leads to a
  full rank lower triangular block matrix with at most $\ell_j$ columns.
  \begin{lemma} \label{lem4:reduce}
   Let condition \eqref{eq4:ordIm} hold. 
   Then the spaces $V,V^{\infty}$ defined above satisfy
  \begin{align*}
    P_{e^{tA}V}-P_{e^{tA}V^{\infty}} \to 0 \quad \text{as} \quad t \to \infty.
  \end{align*}
  \end{lemma}
  \begin{proof}
  Define $B= \mathrm{diag}\big(( \beta_{k_j+1}I_{b_j})_{j=1}^r\big) \in \R^{s,s}$ and observe that the  matrices $U(t)=e^{tA}W e^{- t B}$
  and $U^{\infty}(t) =e^{tA}W^{\infty} e^{- t B}$ satisfy
  \begin{align*}
    \big(U(t)-U^{\infty}(t)\big)_{i,j}= \begin{cases}
      (e^{t \Lambda_{ii}}e^{-t \beta_{k_j+1}})W_{ij}, & i > k_j+\ell_j, \\
      0, & \text{otherwise}.
      \end{cases}
    \end{align*}
  By \eqref{eq4:ordIm} and the ordering in \eqref{eq4:lind}
  we obtain for some constants $C,\gamma>0$
  \begin{align*}
    \|U(t)-U^{\infty}(t)\| \le C e^{-\gamma t}, \quad t \ge 0.
  \end{align*}
  The leading entries of $U^{\infty}(t)$ are 
  \begin{align*}
    U^{\infty}_{k_j+1,j}(t)= \begin{cases}
      W_{k_j+1,j}, & \text{if}\, d_{k_j+1}= 1, \\
      D_{\rho_{k_j+1}}T_{t \omega_{k_j+1}}D_{\rho_{k_j+1}}^{-1} W_{k_j+1,j}, &
      \text{if} \, d_{k_j+1}=2,
    \end{cases}
  \end{align*}
 where $D_{\rho},T_{\omega}$ are defined in \eqref{eq3:introDT}.
  The matrices $ (U^{\infty}(t)^{\top}U^{\infty}(t))^{-1}\in \R^{s,s}$ are uniformly bounded since the entries $W_{k_j+1,j}$ have maximum rank  and the rotations
  $T_{t\omega}$ are uniformly bounded. Further, the matrix
  $S(t) = U(t) (U^{\infty}(t)^{\top}U^{\infty}(t))^{-1}U^{\infty}(t)^{\top}\in \R^{d,d}$ satisfies
  \begin{align*}
    S(t)U^{\infty}(t) = U(t), \quad S(t)e^{tA}W^{\infty}= e^{tA}W, \quad t \ge 0.
  \end{align*}
  For every $v=U^{\infty}(t)e^{t B}b=e^{tA}W^{\infty}b\in e^{tA}V^{\infty}$, $b \in \R^s$ we have the
  estimate
  \begin{align*}
    \|(S(t)-I_d)v\| & = \|S(t)e^{tA}W^{\infty}b - e^{tA}W^{\infty}b \|
      =\|e^{tA} W b - e^{tA}W^{\infty}b\|\\
      & = \|(U(t)-U^{\infty}(t))e^{t B}b \| \le C e^{-\gamma t} \|e^{t B}b\| \\
      & = C e^{-\gamma t}\|(U^{\infty}(t)^{\top}U^{\infty}(t))^{-1}U^{\infty}(t)^{\top}v \| \le Ce^{-\gamma t} \|v\|.
  \end{align*}
  For $t$ sufficiently large this leads to
      \begin{align*}
        \|(I_d-S(t))v\| & \le \frac{C e^{-\gamma t}}{1-C e^{-\gamma t}} \|S(t)v\|
        \quad \forall v \in e^{tA}V^{\infty}.
      \end{align*}
      Thus we can apply Lemma \ref{lem3:2} (ii) and use  \eqref{eq1:relmet}
      to obtain our assertion
            \begin{align*}
              \|P_{e^{tA}V^{\infty}}-P_{e^{tA}V}\| &
              \le \ang(e^{tA}V^{\infty},e^{tA}V) =
              \ang(e^{tA}V^{\infty},S(t)e^{t A}V^{\infty})
                            \le C e^{-\gamma t}.
            \end{align*}
      
  \end{proof}
  The following proposition provides the angular value in two  cases where
  the asymptotics of  $P_{e^{tA}V^{\infty}}$ is easy to analyze.

  \begin{proposition} \label{prop4:realdiag}
      Assume that the real Schur form \eqref{eq4:realschur} of $A$ is
     blockdiagonal with isolated real parts of complex eigenvalues
     as in \eqref{eq4:ordIm}. Then the following assertions hold:
       \begin{itemize}
      \item[(i)]
     If all eigenvalues are real then
     \begin{equation} \label{eq4:0formula}
       \vartheta_s^{\sup,\varlimsup}(A)=\vartheta_s^{\sup,\varliminf}(A)=0.
      \end{equation}
   \item[(ii)] If $J_{\C} $ is nonempty and $s\in\{1,d-1\}$ or
     ($|J_{\C}|=1$, $s <d$) then
      \begin{align} \label{eq4:1formula}
        \vartheta_s^{\sup,\varlimsup}(A)=\vartheta_s^{\sup,\varliminf}(A)=
        \max_{j\in J_{\C}}\omega_j.
      \end{align}
      \end{itemize}
      \end{proposition}
  \begin{remark} Clearly, equation \eqref{eq4:0formula} may be regarded as a special
    case of \eqref{eq4:1formula} by setting $\max_{\emptyset}=0$.
    \end{remark}
  \begin{proof}
    By Proposition \ref{propcont} (iii)  we can assume $A$ to be 
    in Schur normal form \eqref{eq4:realschur}.
    \begin{itemize}
      \item[(i):] In this case we have
        $A = \mathrm{diag}(\lambda_1, \ldots, \lambda_d)$ with $ \lambda_1 \ge \lambda_2\ge \cdots \ge \lambda_d$. Then
      the space $V^{\infty}=\range(W^{\infty})$ defined by  \eqref{eq4:defWinfty}
      is spanned by Cartesian basis vectors, hence $A$ and $e^{tA}$ leave
      $V^{\infty}$ invariant and
      \begin{align*}
        (I_d-P_{e^{tA}V^{\infty}})AP_{e^{tA}V^{\infty}} =(I_d-P_{V^{\infty}})AP_{V^{\infty}}= 0 \quad \forall \  t\ge 0.
      \end{align*}
      Lemma \ref{lem4:reduce} shows $(I_d-P_{e^{tA}V})A P_{e^{tA}V} \to 0$
      and the assertion follows from Definition \ref{contangularvalues}.
        \item[(ii):] Consider first $s=1$. If $k_1+1 \notin J_{\C}$ then $V^{\infty}$
      is invariant under $A$. Hence we have
      $(I_d-P_{e^{tA}V^{\infty}})AP_{e^{tA}V^{\infty}}=0$ and $\lim_{T \to \infty} \frac{1}{T}a_{0,T}(V)=0$ as above.
      If $k_1+1 \in J_{\C}$ then we find from Example \ref{ex4:1}
      \begin{align*}
         \lim_{T \to \infty} \frac{1}{T}a_{0,T}(V^{\infty})
        &= \frac{\rho_{k_1+1} \omega_{k_1+1}}{\pi}\int_0^{\pi}
        \frac{\|W_{k_1+1,1}\|^2}{\|D_{\rho_{k_1+1}}T_tW_{k_1+1,1}\|^2} dt
        = \omega_{k_1+1}.
      \end{align*}
      Taking the maximum over $V\in \cG(1,d)$ then proves \eqref{eq4:1formula}.
      
      In case $1 \le s <d$ and $|J_{\C}|=1$ there are three possibilities.
      If $k_j+1 \notin J_{\C}$ for all $j=1,\ldots,r$ then $V^{\infty}$ is
      invariant under $A$ and $\lim_{T \to \infty} \frac{1}{T}a_{0,T}(V)=0$
      as above. Otherwise, there exists exactly one $j \in \{1,\ldots,r\}$
      with $k_j+1 \in J_{\C}$. If $b_j=2$ then $W_{k_j+1,j}\in \R^{2,2}$
      is invertible and
      \begin{align*}
        \Lambda_{k_j+1,k_j+1}W_{k_j+1,j}= W_{k_j+1,j}M_j, \quad \text{for}
        \; M_j=W_{k_j+1,j}^{-1}\Lambda_{k_j+1,k_j+1}W_{k_j+1,j}.
      \end{align*}
      Hence $V^{\infty}$ is invariant, again leading to the limit zero.
      The final case is $b_j=1$ which occurs at least once since $s<d$.
      Then the only nonzero
      block row of $(I_d-P_{e^{tA}V^{\infty}})AP_{e^{tA}V^{\infty}}$ is given by
      \begin{align} \label{eq4:relV}
        \big(I_2-P_{V_j(t)}\big)\Lambda_{k_j+1,k_j+1}P_{V_j(t)}, \quad
        \text{where} \; V_j(t)= \mathrm{span}(e^{t \Lambda_{k_j+1,k_j+1}}W_{k_j+1,j}).
      \end{align}
      The spectral norm of this $2\times 2$ matrix yields the spectral
      norm of $(I_d-P_{e^{tA}V^{\infty}})AP_{e^{tA}V^{\infty}}$ and leads to
      $ \lim_{T \to \infty} \frac{1}{T}a_{0,T}(V^{\infty})= \omega_{k_j+1}$ as
      above. This proves \eqref{eq4:1formula}.

      Finally, let $s=d-1$. Since $\sum_{j=1}^rb_j=s$ for the
      column echelon form \eqref{eq4:columnechelon} there is exactly
      one index $j \in \{1,\ldots,r\}$ with $b_j=1$, $d_{k_j+1}=2$ while
       all $b_{\nu}$, $\nu \neq j$ attain their maximum value
      (either $1$ or $2$). As above, the only nonzero 
       block row of $(I_d-P_{e^{tA}V^{\infty}})AP_{e^{tA}V^{\infty}}$ is given by
       \eqref{eq4:relV} and we obtain $ \lim_{T \to \infty} \frac{1}{T}a_{0,T}(V^{\infty})= \omega_{k_j+1}$ and then formula \eqref{eq4:1formula} by maximizing over $V\in \cG(d-1,d)$.
           \end{itemize}  
  \end{proof}
  Next we consider the case of general dimension $s$ and arbitrary index set $J_{\C}$. As we have seen above, only the leading entries $W_{k_j+1,j}$ with $b_j=1$
  contribute to the angular value. Therefore, we determine
  those index sets $J \subseteq J_{\C}$ which belong to elements $V \in \cG(s,d)$
  with $b_j=1$ for $j \in J$. We call these sets admissible and show that
  they are given by
  \begin{equation} \label{eq4:defcJ}
    \cJ(s,\C)= \{ J \subseteq J_{\C}:|J| \le \min(s,d-s) \quad \text{and} \quad
    (s-|J|\; \text{is even, if}\; |J_{\C}|=k)\}.
    \end{equation}
  \begin{lemma} \label{lem4:detV}
    A subset $J\subseteq J_{\C}$ belongs to\ $\cJ(s,\C)$ if and only if there exists
    $V\in \cG(s,d)$ with $V =\range(W)$ such that the  
    column echelon form \eqref{eq4:columnechelon}  of $W$ has 
    $b_j=1$, $d_{k_j+1}=2$ exactly for the indices $j \in J$.
  \end{lemma}
    \begin{proof} First assume that an appropriate $V \in \cG(s,d)$ exists.
    Then the estimate $|J|\le s$ follows from \eqref{eq4:columnechelon},
    since the columns which belong to $j\in J$ are linearly independent.
    By dimension counting we further have $s=|J|+n_1+ 2n_2$
    where $0 \le n_1 \le k-|J_{\C}|$ (number of real eigenvalues) and
    $ 0 \le n_2 \le |J_{\C}|- |J|$ (number of remaining complex eigenvalues).
    If $k=|J_{\C}|$ then $n_1=0$ and $s- |J|\ge 0$ is even. In general, we have
    \begin{equation} \label{eq4:indexrel}
      s\le  |J|+ k - |J_{\C}|+ 2 (|J_{\C}|- |J|)
      =k+ |J_{\C}|- |J| = d - |J|,
    \end{equation}
    whence $|J| \le  d-s$.
    Conversely, assume $J \in \cJ(s,\C)$ so that the inequality
    \eqref{eq4:indexrel} holds.
    If $k = |J_{\C}|$ then $s-|J|$ is even 
    and we have $s=|J|+ 2 n_2$, where $0 \le n_2 \le |J_{\C}|-|J|$ follows
    from \eqref{eq4:indexrel}. If $k -|J_{\C}|\ge 1$ holds, we set
    $n_2= \min(\lfloor \frac{s - |J|}{2}\rfloor,|J_{\C}|- |J|)\ge 0$ and
    $n_1= s - |J|- 2 n_2$. Then we have $n_1 \ge s - |J|- 2\lfloor \frac{s - |J|}{2}\rfloor \ge 0$. Further, if $n_2 =\lfloor \frac{s - |J|}{2}\rfloor$
    we conclude $n_1\in \{0,1\}$, hence $n_1 \le k-|J_{\C}|$.
    In case $n_2 = |J_{\C}|-|J|$ we find from \eqref{eq4:indexrel}
    \begin{align*}
      n_1 & \le k- |J_{\C}| + 2(|J_{\C}|- |J|)- 2 n_2=k-|J_{\C}|,
    \end{align*}
    as required. Summarizing, we can construct some $V=\range(W) \in
    \cG(s,d)$ with a matrix $W$ of rank $s$
    as in \eqref{eq4:columnechelon} with $|J|$ columns belonging to indices in
    $J_{\C}$, $n_2$ double columns  belonging to the remaining complex
    eigenvalues, and $n_1$ columns belonging to real eigenvalues.
    \end{proof}
  
 \begin{proposition} \label{prop4:autonomous}
     Assume that the real Schur form \eqref{eq4:realschur} of $A$ is
     blockdiagonal with nonempty $J_{\C}$ and isolated real parts
     as in \eqref{eq4:ordIm}. Then the following holds:
     \begin{itemize}
     \item[(i)] For $1\le s \le d-1$,
       $\vartheta_s^{\sup,\varlimsup}(A)=\vartheta_{d-s}^{\sup,\varlimsup}(A)$,
       $\vartheta_s^{\sup,\varliminf}(A)=\vartheta_{d-s}^{\sup,\varliminf}(A)$.
       \item[(ii)]
     If the values $\omega_j$, $j\in J$ are rationally independent for every
     $J\in \cJ(s,\C)$, then the outer angular values are given by
      \begin{equation} \label{eq4:sformula}
        \begin{aligned}
        \vartheta_s^{\sup,\varlimsup}(A)&=\vartheta_s^{\sup,\varliminf}(A)
        = \max_{J\in \cJ(s,\C)}\frac{1}{\pi^{|J|}}
        \int_{[0,\pi]^{|J|}} \max_{j \in J} E_j(\tau_j) \,  d(\tau_j)_{j \in J}, \\
        E_j(\tau_j)& = \frac{\rho_j \omega_j}{\cos^2(\tau_j) + \rho_j^2 \sin^2(\tau_j)}. 
        \end{aligned}
      \end{equation}
      \end{itemize}
 \end{proposition}
 
 \begin{remarks}\label{rem4:blockschur} { \ } 
   \begin{itemize}
     \item[(a)]
    Recall from \cite[Prop.1.4.1]{KH95} that numbers $a_1,\ldots, a_r \in \R$
   are rationally independent if $\sum_{j=1}^r m_ja_j = 0$ for
   some $m_j\in \Z$ implies $m_1=\cdots = m_r=0$. In particular, in the
   simple case  $s=1$ or  $|J_{\C}|=1$ each single value $\omega_j \neq 0$ is
   rationally independent so that \eqref{eq4:1formula} becomes a special
   case of the formula \eqref{eq4:sformula}. In case $s=d$ we have
     $\vartheta_s^{\sup,\lim}(A)=0$ which is consistent with \eqref{eq4:sformula}
   when setting $\max_{\emptyset}=0$.
   \item[(b)]
   Estimating $\max_{j\in J}E_j \ge E_{\ell}$ for each $\ell \in J$ in  \eqref{eq4:sformula} 
   leads to the lower bound
   $\vartheta_s^{\sup,\lim}(A) \ge \max_{j\in J_{\C}}\omega_j$, since  for every $j \in J_{\C}$ there is an element $J \in \cJ(s,\C)$ with $j \in J$.
   Similarly, the integral expression in \eqref{eq4:sformula} grows
   when passing from $J'\in \cJ(s,\C)$ to some $J\in \cJ(s,\C)$ with $J'\subset J$.
   Therefore, it suffices to consider  only maximal index sets
   $J \in \cJ(s,\C)$ in \eqref{eq4:sformula}.
    \item[(c)]
    Recall that normal matrices have a block diagonal real Schur form \cite[Theorem 2.5.8]{HJ2013}. However, $\Lambda_{jj}$ in \eqref{eq4:realschur} is not
    normal if $\rho_j<1$. Hence, Propositions \ref{prop4:realdiag},
    \ref{prop4:autonomous} cover a class of nonnormal matrices.
    \item[(d)]
    Let us finally mention that  Propositions \ref{prop4:realdiag}, \ref{prop4:autonomous} can be extended
    in several ways.
First,  the inner angular values 
$\vartheta_s^{\varlimsup,\sup}(A)$, $\vartheta_s^{\varliminf,\sup}(A)$
agree with the outer ones in both assertions.  Second,
one can treat normal forms with nonzero blocks above the diagonal in \eqref{eq4:realschur}.   A proof of these results needs reduction techniques analogous
to those of the discrete case; see \cite[Section 5]{BeFrHu20}, \cite[Section 3]{BeHu22}. Moreover, relaxing the assumption on rational independence requires
further tools, such as unique ergodicity and the ergodic
decomposition theorem (\cite[Ch.4.1,4.2]{KH95}, \cite[Theorem 3.3]{Ba12}).
The corresponding results and their proofs are beyond the scope of this paper.
\end{itemize}
 \end{remarks}

\begin{example} \label{Ex4:2prop} (Example \ref{Ex4:2} revisited)\\
  From \eqref{eq4:indexchoices} 
  we infer that the admissible index sets are given by
   \begin{align*}
     \cJ(1,\C)  = \{ \{1\}, \{2\}\}, \quad
     \cJ(2,\C) = \{ \{1,2\} \}, \quad 
     \cJ(3,\C)  = \{ \{1\}, \{2\}\}.
   \end{align*}
   By Proposition \ref{prop4:realdiag} (ii) we have
   $\vartheta_{1}^{\sup,\lim}(A) = \max(\omega_1,\omega_2)=\vartheta_{3}^{\sup,\lim}(A)$.
   At first glance, this equality may come as a surprise, but Proposition
   \ref{prop4:autonomous} (i) shows that this symmetry holds in general.
   For this example, subspaces of dimension $s=2$ are the most interesting ones since they have more
   flexibility to rotate.
   Figure \ref{maxgraph} shows the graph of the integrand
   in \eqref{eq4:sformula} for the case $s=2=|J_{\C}|$ with the rationally
   independent values $\omega_1=1$, $\omega_2=\frac{1}{\sqrt{2}}$.
   Numerical integration yields the angular value $\vartheta_2^{\sup,\lim}(A)=
      1.2693394 >\max(\omega_1,\omega_2)=1$.
 \end{example}
 
\begin{figure}[hbt]
  \begin{center}
    \includegraphics[width=0.50\textwidth]{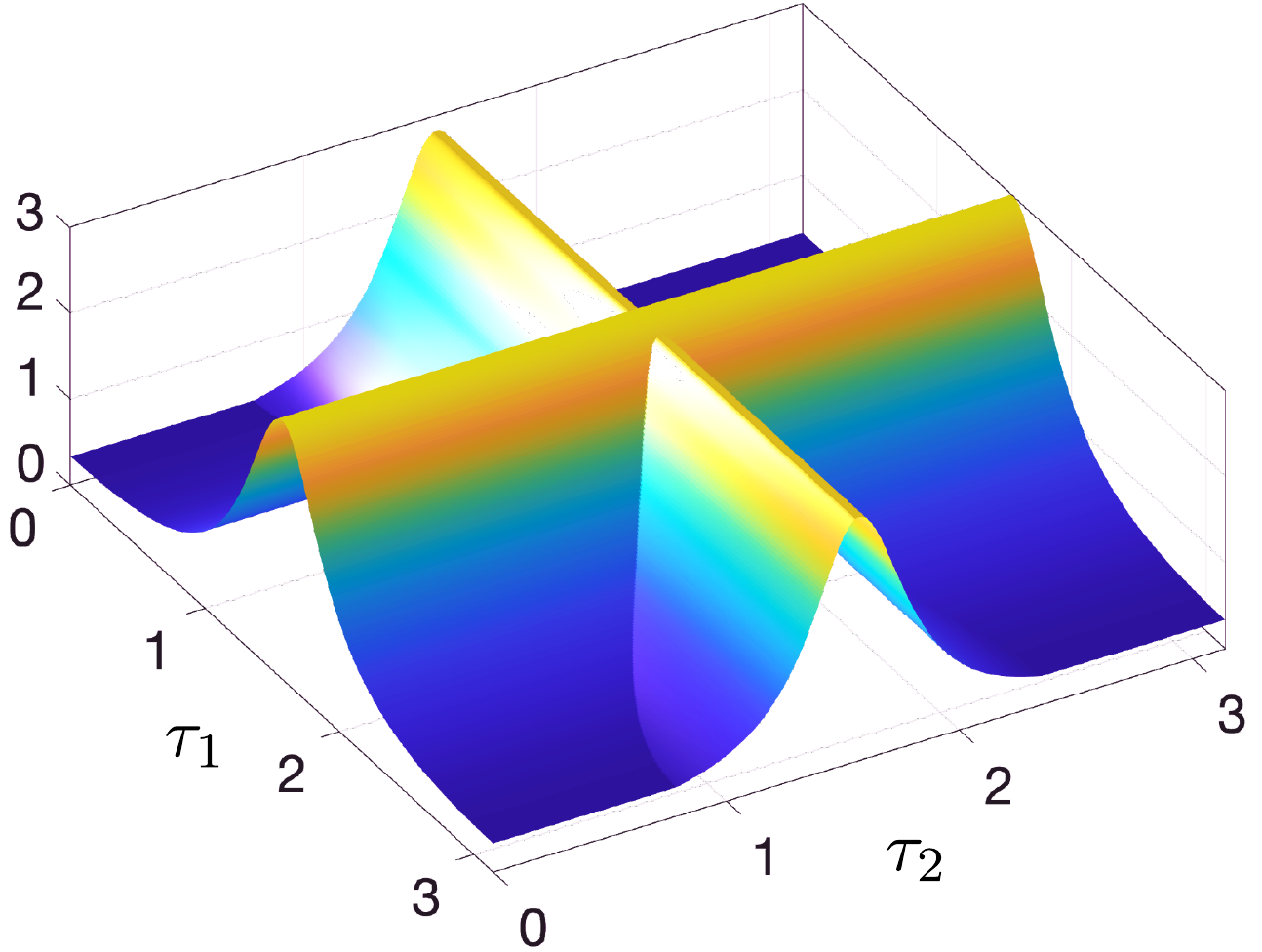}
    \end{center}
  \caption{\label{maxgraph} Graph of the integrand from \eqref{eq4:sformula} for the two-dimensional case $J=\{1,2\}$. Parameter values: $\rho_1=\frac{1}{3}$, $\omega_1=1$, $\rho_2=\frac{1}{4}$,
  $\omega_2 = \frac 1{\sqrt{2}}$.} 
\end{figure}

  \begin{proof}
    By Proposition \ref{propcont} (ii) we can assume $A$ itself to be in Schur normal form \eqref{eq4:realschur}  and by Lemma \ref{lem4:reduce} it suffices
    to compute the angular value for $V^{\infty}=\range(W^{\infty})$ from \eqref{eq4:defWinfty}.
        The projections $P_{e^{tA}V^{\infty}}$ are direct products of projections
      which either belong to a simple complex eigenvalue or to a group
      of identical real eigenvalues. As we have seen above, the latter ones
      and the projections onto $(0 \cdots 0\ W_{k_j+1,j}\ 0 \cdots 0)^{\top}$
      with $d_{k_j+1}=2$, $b_j=2$ lead to zero terms in
      $(I_d-P_{e^{tA}V^{\infty}})AP_{e^{tA}V^{\infty}}$. Hence, the spectral norm
      is determined by index sets $J \in \cJ(s,\C)$  and their corresponding
      elements
      in $\cG(s,d)$ from Lemma \ref{lem4:detV}.
      More specifically, the spectral norm is computed from the
      vectors $W_{k_j+1,j}\in \R^2$, $j\in J$ and the spaces 
      $V_j(t)$, $j \in J$ in \eqref{eq4:relV} as follows:
      \begin{equation} \label{eq4:specnorm}
        \begin{aligned}
        \|(I_d-P_{e^{tA}V^{\infty}})AP_{e^{tA}V^{\infty}})\|&= \max_{j \in J}
        \|(I_2-P_{V_j(t)})\Lambda_{k_j+1,k_j+1}P_{V_j(t)}\|.
        \end{aligned}
      \end{equation}
      \begin{itemize}
      \item[(i)] This assertion follows from \eqref{eq4:specnorm} and the equality
        $\cJ(s,\C)=\cJ(d-s,\C)$ for the sets in \eqref{eq4:defcJ}.
       To see this, note that the condition
        $|J|\le \min(s,d-s)$ is symmetric. Further, if  $k=|J_{\C}|$ holds then 
        $d=2 k$ is even, hence $s-|J|$ and $d-s-|J|$ have the same parity.
      \item[(ii)]
          In the following we write the function $\alpha$ from \eqref{eq3:angformula} as 
      \begin{equation} \label{eq4:alphadef}
        \alpha(\tau,v,\rho,\omega)= 
         \frac{\rho \omega \|D_{\rho}^{-1}v\|^2}{\|D_{\rho}T_{\tau \omega}
           D_{\rho}^{-1}v\|^2}, \quad \tau\in \R, v\in
         \R^2\setminus\{0\},\ 0<\rho \le 1,\ 
         \omega>0.
      \end{equation}
      The function has  the following properties
      \begin{equation}\label{eq4:alphaprop}
        \begin{aligned}
          \alpha(\tau+\tfrac{2 \pi}{\omega},v,\rho,\omega)& = \alpha(\tau,v,\rho,\omega), \\
          \alpha(\tau,xv,\rho,\omega)&= \alpha(\tau,v,\rho,\omega), \quad x \in \R,\\
        \alpha(\tau,D_{\rho}T_{t\omega}D_{\rho}^{-1}v,\rho,\omega) &
        = \alpha(\tau+t,v,\rho,\omega), \quad t \in \R.
        \end{aligned}
      \end{equation}
For $J \subseteq J_{\C}$ we introduce the $|J|$-dimensional tori 
      \begin{equation*} \label{eq4:toridef}
        \begin{aligned}
          \T^J_{1}& = \{ v_J=(v_j)_{j \in J}: v_j \in \R^2,\ \|v_j\|=1\ (j\in J) \}, \\
          \T_{\rho}^J& = \{ D_{[\rho]}v_J: v_J\in
          \T^{J}_{1}\}, \quad \text{where} \;D_{[\rho]}(v_J)= (D_{\rho_j}v_j)_{j \in J}.
        \end{aligned}
              \end{equation*}
            From \eqref{eq4:specnorm} and \eqref{eq3:angformula} we obtain
      \begin{equation*} \label{eq4:angleformula}
        \vartheta_s^{\sup, \lim}(A)= \sup_{J\in \cJ(s,\C)}\;
        \sup_{v_J \in (\R^2\setminus\{0\})^{|J|}} \; \lim_{T \to \infty}\frac{1}{T}
        \int_0^T \max_{j \in J}\alpha(\tau,v_j,\rho_j,\omega_j) d\tau.
      \end{equation*}
      We show that the limit in this expression exists and
      yields the formula \eqref{eq4:sformula}.
      
      For a given nonempty $J \in \cJ(s,\C)$ we choose
      some $\ell \in J$ and set $\cL= J \setminus \{\ell\}$. Since the function $\alpha$ is bounded it is sufficient to consider the limit as
      $T=N\tau_{\ell} \to \infty$ where $N \in \N$ and
      $\tau_{\ell}=\frac{2 \pi}{\omega_{\ell}}$: \allowdisplaybreaks[3]
      \begin{align*}
        & \frac{1}{T} \int_0^T \max_{j \in J}\alpha(\tau,v_j,\rho_j,\omega_j) d\tau
        = \frac{1}{N\tau_{\ell}} \sum_{i=1}^{N} \int_{0}^{\tau_{\ell}}
        \max_{j \in J}\alpha((i-1)\tau_{\ell}+t,v_j,\rho_j,\omega_j) dt \\
        & = \frac{1}{N\tau_{\ell}} \sum_{i=1}^{N} \int_{0}^{\tau_{\ell}} \max\Big(
        \alpha(t,v_{\ell},\rho_{\ell},\omega_{\ell}),
        \max_{j \in \cL}\alpha((i-1)\tau_{\ell}+t,v_j,\rho_j,\omega_j)\Big) dt\\
        & = \frac{1}{N} \sum_{i=1}^N g(v_{\ell},F_{\rho}^{i-1} v_{\cL}).
      \end{align*}
      Here the function $g:\T_{\rho}^{J}\to \R $ and the map $F_{\rho}:\T_{\rho}^{\cL}\to
      \T_{\rho}^{\cL}$ are defined by
      \begin{equation} \label{eq4:defgF}
        \begin{aligned}
          F_{\rho}(v_{\cL}) &  = (D_{\rho_j}T_{\omega_j \tau_{\ell}} D_{\rho_j}^{-1}v_j)_{j \in \cL},
          \\
          g(v_{\ell}, v_{\cL})& = \frac{1}{\tau_{\ell}}\int_0^{\tau_{\ell}}\max\Big(\alpha(t,v_{\ell},\rho_{\ell},\omega_{\ell}), \max_{j \in \cL} \alpha(t,v_j,\rho_j,\omega_j)\Big) dt.
        \end{aligned}
      \end{equation}
      Further, we used the properties \eqref{eq4:alphaprop} which imply that
      $g$ is homogeneous in its $v$-variables.
      Our goal is to apply Birkhoff's ergodic theorem (see e.g.\
      \cite[Theorem 2.2]{Ba12}) to the map $F_{\rho}$
      and the continuous function $g(v_{\ell}, \cdot)$ for any fixed $v_{\ell} \in \R^2
      \setminus \{0\}$.  From \cite[Prop.4.2.2]{KH95} we have that the map
      \begin{align*}
        F_{1}: \T^{\cL}_{1}\to \T^{\cL}_{1}, \quad u_{\cL} \mapsto (T_{\omega_j \tau_{\ell}}u_j)_{j \in \cL}
      \end{align*}
      is ergodic 
      with respect to Lebesgue measure $\mu_{1}$ on the standard torus $\T^{\cL}_{1}$
      if the numbers $(\omega_j \tau_{\ell}=
      \frac{2 \pi \omega_j}{\omega_{\ell}})_{j\in \cL}$ and $\pi$ are rationally
      independent. This holds by our assumption.
      It is then easy to verify that the topologically conjugate map
      $F_{\rho}=D_{[\rho]}F_{1} D_{[\rho]}^{-1}$ is ergodic on $\T^{\cL}_{\rho}$
      with respect to the image measure $\mu_{\rho}=D_{[\rho]}\circ \mu_{1}$.
      Recall from \cite[Definition 7.6]{Bau01} the definition
       $ \mu_{\rho}(B) = \mu_{1}(D_{[\rho]}^{-1}B)$ for Borel sets $B$ in
      $ \T^{\cL}_{\rho}$. We conclude
      that the following limit exists for each $v_{\ell}\neq 0$ and is independent of $v_{\cL} \in \T_{\rho}^{\cL}$
      \begin{equation} \label{eq4:limerg}
      \begin{aligned}
        \lim_{T\to \infty} \frac{1}{T} \int_0^T \max_{j \in J}\alpha(\tau,v_j,\rho_j,\omega_j) d\tau & =\frac{1}{\mu_{\rho}(T_{\rho}^{\cL})} \int_{T_{\rho}^{\cL}}g(v_{\ell},u_{\cL}) d\mu_{\rho}(u_{\cL})\\
        &=\frac{1}{\mu_{1}(T_{1}^{\cL})} \int_{\T^{\cL}_{1}} g(v_{\ell},D_{[\rho]}u_{\cL})d\mu_{1}(u_{\cL}).
      \end{aligned}
      \end{equation}
      The last equality follows from the transformation formula \cite[Theorem 19.1]{Bau01}.
      Next, we evaluate the integral and take the supremum w.r.t.\ $v_{\ell}$:
      \begin{align*}
        & \sup_{v_{\ell}\neq 0}\frac{1}{\mu_{1}(T_{1}^{\cL})} \int_{\T^{\cL}_{1}} g(v_{\ell},D_{[\rho]}u_{\cL})d\mu_{1}(u_{\cL}) \\
        & = \sup_{v_{\ell}\neq 0}\frac{1}{\tau_{\ell}(2 \pi)^{|\cL|}} \int_{\T^{\cL}_{1}} \int_0^{\tau_{\ell}} \max\Big(\alpha(t,v_{\ell},\rho_{\ell},\omega_{\ell}),
        \max_{j \in \cL} \alpha(t,D_{\rho_j}u_j,\rho_j,\omega_j) \Big) dt
        d\mu_{1}(u_{\cL})\\
        & = \sup_{v_{\ell}\neq 0}\frac{1}{(2 \pi)^{|J|}} \int_{\T^{\cL}_{1}}
        \int_0^{2 \pi} \max\Big(\alpha(\tfrac{\tau}{\omega_{\ell}},v_{\ell},\rho_{\ell},\omega_{\ell}),
        \max_{j \in \cL} \alpha(\tfrac{\tau}{\omega_{\ell}},D_{\rho_j}u_j,\rho_j,\omega_j) \Big) d\tau
        d\mu_{1}(u_{\cL}).
      \end{align*}
      Then we use Fubini's theorem, the homogeneity of $\alpha$ w.r.t.\ $v$ from \eqref{eq4:alphaprop} and substitute $(T_{\tau \frac{\omega_j}{\omega_{\ell}}}u_j)_{j\in \cL}$
      by $u_{\cL}$ on $\T^{\cL}_1$:  \allowdisplaybreaks[3]
      \begin{align*}
        & \sup_{v_{\ell}\neq 0}\frac{1}{\mu_{1}(T_{1}^{\cL})} \int_{\T^{\cL}_{1}} g(v_{\ell},D_{[\rho]}u_{\cL})d\mu_{1}(u_{\cL}) \\
        & = \sup_{\|v_{\ell}\|=1}\frac{1}{(2 \pi)^{|J|}} \int_0^{2 \pi} \int_{\T^{\cL}_{1}}
         \max\Big(\frac{\rho_{\ell}\omega_{\ell}}{\|D_{\rho_{\ell}}T_{\tau} v_{\ell}\|^2},
        \max_{j \in \cL} \frac{\rho_j \omega_j}{\|D_{\rho_j}T_{\tau \tfrac{\omega_j}{\omega_{\ell}}} u_j \|^2} \Big) 
        d\mu_{1}(u_{\cL})d\tau
        \\
        & =\sup_{\|v_{\ell}\|=1}\frac{1}{(2 \pi)^{|J|}} \int_0^{2 \pi} \int_{\T^{\cL}_{1}}
         \max\Big(\frac{\rho_{\ell}\omega_{\ell}}{\|D_{\rho_{\ell}}T_{\tau} v_{\ell}\|^2},
        \max_{j \in \cL} \frac{\rho_j \omega_j}{\|D_{\rho_j} u_j \|^2} \Big) 
        d\mu_{1}(u_{\cL})d\tau \\
        & 
        =\sup_{\|v_{\ell}\|=1}\frac{1}{(2 \pi)^{|J|}} \int_{\T_1^{\{\ell\}}} \int_{\T^{\cL}_{1}}
         \max\Big(\frac{\rho_{\ell}\omega_{\ell}}{\|D_{\rho_{\ell}}u_{\ell}\|^2},
        \max_{j \in \cL} \frac{\rho_j \omega_j}{\|D_{\rho_j} u_j \|^2} \Big) 
        d\mu_{1}(u_{\cL})d\mu_1(u_{\ell})\\
        & = \frac{1}{(2 \pi)^{|J|}}  \int_{\T^{J}_{1}}
        \max_{j \in J} \Big( \frac{\rho_j \omega_j}{\|D_{\rho_j} u_j \|^2} \Big) 
        d\mu_{1}(u_{J}).
             \end{align*}
      For the last steps note that the supremum  becomes  obsolete since
      $T_{\tau}v_{\ell}$, $\tau \in [0,2\pi]$ parameterizes the circle
      $\T_1^{\{\ell\}}$ for
      every $v_{\ell}\in \T_1^{\{\ell\}}$. Finally, this proves  formula
      \eqref{eq4:sformula} when parameterizing $\T^{J}_{1}$ and replacing
      $2 \pi$ by the shorter period $\pi$ of the integrand.
      \end{itemize}
            \end{proof}
  
  \subsection{Upper and lower semicontinuity of angular values}
  \label{sec4.4}
  As mentioned in the introduction, continuity of angular values
  with respect to system parameters is a delicate matter.
  In particular, we proved in \cite[Section 6.1]{BeFrHu20} that the outer
  angular value for the discrete autonomous case
  \begin{equation} \label{eq4:discA}
    u_{n+1}= A(\rho,\varphi) u_n,\quad u_0 \in \R^2, \quad
    A(\rho,\varphi)= \begin{pmatrix} \cos(\varphi) & -\frac{1}{\rho}\sin(\varphi)\\ \rho \sin(\varphi) & \cos(\varphi) \end{pmatrix}= D_{\rho}T_{\varphi}D_{\rho}^{-1}
  \end{equation}
      is not lower semicontinuous (lsc) with respect to $\varphi$ at values
    $\varphi= \frac{\pi}{q}$, $q \ge 2$. Note that $A(\rho,\varphi)$
      is the time $1$-flow of our Example \ref{ex3:1}.
      Therefore, one can expect at most upper semicontinuity (usc)
      of angular values. It is the purpose of this subsection
      to establish such a result for two model examples: the first angular value of the
      discrete system \eqref{eq4:discA} and the second angular value
      of the four-dimensional model system \eqref{eq4:A4}.    
      
      In the following let $\T=S^1$ be the unit circle in $\R^2$
      and let $M$ be a subset of some $\R^m$. Our tool
      for both examples is the following result:
      \begin{proposition} \label{prop4:baserg}
        For $f\in C(\T \times [0,2\pi]\times M,\R)$ define $f_{\infty}:\T \times [0,2\pi]\times M \to \R$ by
        \begin{equation} \label{eq4:deftheta}
          f_{\infty}(x,\varphi,\lambda) = \begin{cases}
            \frac{1}{2 \pi} \int_{\T} f(\xi,\varphi,\lambda) d\mu_1(\xi),
            & \text{if} \; \varphi,\pi \; \text{are rationally independent}, \\
            \frac{1}{q}\sum_{j=0}^{q-1} f(T^j_{\varphi}x,\varphi, \lambda), &
            \text{if} \; \frac{\varphi}{2 \pi}= \frac{p}{q},\ p,q \in \N,\
            p \perp q.
          \end{cases}
        \end{equation}
        Then the function
        \begin{equation*} \label{eq4:defsupfunc}
          \theta_{\infty}:[0,2\pi]\times M \to \R, \quad
          \theta_{\infty}(\varphi,\lambda)= \sup_{x \in \T} f_{\infty}(x,\varphi,\lambda)
        \end{equation*}
        is upper semicontinuous (usc) and continuous at points
        $(\varphi,\lambda)
        \in (0,2 \pi)\times M$ with $\varphi \notin \pi \Q$.
      \end{proposition}
      \begin{remark} In definition \eqref{eq4:deftheta}  we set $p=0$, $q=1$
        if $\varphi=0$ and $p=q=1$ if $\varphi=2\pi$.
        For a fixed $\varphi\in [0,2\pi]$ it is clear that
        $f_{\infty}$ is continuous w.r.t.\ $(x,\lambda)$ and $\theta_{\infty}$
        is continuous w.r.t.\ $\lambda$. The main result above is upper semicontinuity
        of $\theta_{\infty}$ w.r.t.\ $\varphi$ for fixed $\lambda$. However, combining these
        two properties is not enough to establish joint upper semicontinuity.
      \end{remark}
      \begin{proof}
        It suffices to show that every sequence $(\varphi_n,\lambda_n)_{n\in \N}$
         converging to some $(\varphi,\lambda) \in [0,2\pi]\times M$ has
        a subsequence $\N'\subseteq \N$ such that
        $\varlimsup_{\N'\ni n \to \infty}\theta_{\infty}(\varphi_n,\lambda_n) \le
        \theta_{\infty}(\varphi,\lambda)$.
        In case
        $\varphi= 2 \pi \frac{p}{q}$ for some $ p,q \in \N$ with
        $ p \perp q$ it is enough to analyze two subcases:
        \begin{itemize}
        \item[(i)] $\varphi_n \notin 2 \pi \Q$ for all $n \in \N$: \;
          Then we conclude for all $x \in \T$
          \begin{align*}
            f_{\infty}(x,\varphi_n,\lambda_n)& = \frac{1}{2 \pi} \int_{\T}
            f(\xi,\varphi_n,\lambda_n) d\mu_1(\xi)
            = \frac{1}{2 \pi q} \sum_{\nu=0}^{q-1} \int_{\T}f(T_{\frac{2 \nu\pi}{q}}\xi, \varphi_n,\lambda_n) d\mu_1(\xi) \\
            & = \frac{1}{2 \pi q}\int_{\T} \sum_{j=0}^{q-1}
            f(T_{\frac{2 jp\pi}{q}}\xi, \varphi_n,\lambda_n) d\mu_1(\xi) 
            .
          \end{align*}
          In the last step we used $p \perp q$ which implies that
          for every $\nu \in \{0,\ldots,q-1\}$ there exists a
          unique index $j\in \{0,\ldots,q-1\}$ with $\nu=j p \, \mathrm{mod}\; q$.
          This step is trivial if $q=1$. Then we continue with the estimate
          \begin{align*}
            f_{\infty}(x,\varphi_n,\lambda_n)& \le
            \frac{1}{q} \sup_{\xi \in \T} \sum_{j=0}^{q-1} f(T_{\frac{2 j p \pi}{q}}\xi, \varphi_n,\lambda_n)
             = \frac{1}{q} \sup_{\xi \in \T} \sum_{j=0}^{q-1} f(T_{\varphi}^j\xi, \varphi_n,\lambda_n).
          \end{align*}
          Taking the $\sup$ over $x \in \T$ and $\varlimsup_{n \to \infty}$
          on both sides leads to
          \begin{align*}
       \varlimsup_{n \to \infty}\theta_{\infty}(\varphi_n,\lambda_n) &=     \varlimsup_{n \to \infty}\sup_{x \in \T}f_{\infty}(x,\varphi_n,\lambda_n) \le
            \lim_{n \to \infty} \frac{1}{q} \sup_{\xi \in \T} \sum_{j=0}^{q-1} f(T_{\varphi}^j\xi, \varphi_n,\lambda_n)\\
            & = \frac{1}{q} \sup_{\xi \in \T} \sum_{j=0}^{q-1} f(T_{\varphi}^j\xi, \varphi,\lambda) = \theta_{\infty}(\varphi,\lambda).
          \end{align*}
                  \item[(ii)] $\varphi_n= 2 \pi \frac{p_n}{q_n}$ for some $p_n,q_n \in \N$ with $p_n\perp q_n$, $n \in \N$:\\
                    If $(q_n)_{n\in \N}$ has a bounded subsequence then
                    $\varphi_n\to \varphi$ and  $p_n\perp q_n$ imply
     $p_n=p$, $q_n=q$ for $n$ sufficiently large. In this case our assertion
     is trivial. Hence we can assume $q_n \to \infty$. 
     By the continuity of $f$ there exists for every $\varepsilon>0$ some
     $\delta,C>0$ such that for all $ x\in \T$, $\tilde{\varphi}\in [0,2\pi]$
       with
     $|\tilde{\varphi}-\varphi|\le \delta$, and for all $\tilde{\lambda}\in M$ with
      $ \|\tilde{\lambda}-\lambda\| \le \delta$ the following holds
     \begin{equation} \label{eq4:gest}
       \frac{1}{q} \sum_{\ell=0}^{q-1} f(T_{\tilde{\varphi}}^{\ell}x,
       \tilde{\varphi},\tilde{\lambda}) \le \theta_{\infty}(\varphi,\lambda) + \frac{\varepsilon}{2}, \quad
       |f(x,\tilde{\varphi},\tilde{\lambda})| \le C.
         \end{equation}
     Take $N \in \N$ such that $|\varphi_n-\varphi|,|\lambda_n-\lambda|\le \delta$
     and $\frac{qC}{q_n} \le \frac{\varepsilon}{2}$
     for $n\ge N$. Then
     we decompose $q_n=k_nq+r_n$, $k_n \in \N_0$, $0 \le r_n <q$
     and estimate with \eqref{eq4:gest} for all $x \in \T$ as follows:
     \begin{align*}
       f_{\infty}(x,\varphi_n,\lambda_n)& = \frac{1}{q_n} \sum_{j=0}^{q_n-1}
       f(T_{\varphi_n}^j x,\varphi_n,\lambda_n) \\
       & = \frac{1}{q_n}\Big[ \sum_{\nu=0}^{k_n-1} \sum_{\ell=0}^{q-1}
         f(T_{\varphi_n}^{\ell}(T_{\nu q \varphi_n}x),\varphi_n,\lambda_n)
         + \sum_{\ell=0}^{r_n} f(T_{(k_n q + \ell)\varphi_n},\varphi_n,\lambda_n)
         \Big]\\
       & \le \frac{1}{q_n}\Big[ \sum_{\nu=0}^{k_n -1} q(\theta_{\infty}(\varphi,\lambda)+ \frac{\varepsilon}{2}) + q C \Big]
       \le \theta_{\infty}(\varphi,\lambda) + \varepsilon.
     \end{align*}
     Taking the $\sup$ over $x\in\T$ and the $\varlimsup_{n \to \infty}$ yields
     our assertion.
\end{itemize}
     It remains to prove continuity at $\varphi \notin  \pi \Q$.
     If  $\varphi_n \notin \pi \Q$ holds for all $n \in \N$ then we 
     have 
     \begin{align*}
       \lim_{n\to \infty}  f_{\infty}(x,\varphi_n,\lambda_n)& =
       \lim_{n\to \infty}\frac{1}{2 \pi}
       \int_{\T} f(\xi,\varphi_n,\lambda_n) d\mu_1(\xi) \\
       & = \frac{1}{2 \pi}
       \int_{\T} f(\xi,\varphi,\lambda) d\mu_1(\xi)= f_{\infty}(x,\varphi,\lambda)
     \end{align*}
     uniformly in $x \in \T$. Hence, also $\lim_{n \to \infty}\theta_{\infty}(\varphi_n,\lambda_n)= \theta_{\infty}(\varphi,\lambda)$ holds.
     Finally consider $\varphi_n = 2 \pi \frac{p_n}{q_n}$ with $p_n\perp q_n$
     for $n \in \N$. As above we conclude $q_n \to \infty$ because
     otherwise $p_n,q_n$ will be eventually constant and $\varphi= \lim_{n\to \infty}\varphi_n$ will be a rational multiple of $\pi$.
     Then the continuity of $f$ guarantees convergence (uniformly in $x \in \T$)
     of the Riemann sums with stepsize $\frac{2 \pi}{q_n}$:
     \begin{align*}
       \frac{2 \pi}{q_n}\sum_{\nu=0}^{q_n -1}f(T_{\varphi_n}^{\nu}x, \varphi_n,\lambda_n) & = \frac{2 \pi}{q_n} \sum_{j=0}^{q_n-1} f(T_{j \frac{2 \pi}{q_n}}x,\varphi_n,\lambda_n)\\
       & \stackrel{n \to \infty}{\xrightarrow{\hspace*{1cm}}} \int_0^{2 \pi} f(T_{\psi}x, \varphi,\lambda) d\psi
       = \int_{\T}f(\xi,\varphi,\lambda) d\mu_1(\xi).
     \end{align*}
       Note that for every $j \in \{0,\ldots,q_n-1\}$
     there is a unique $\nu \in \{0,\ldots,q_n-1\}$ that solves
     $\nu p_n = j\, \mathrm{mod}\, q_n$.
     Dividing by $2 \pi$ and taking the supremum over $x \in \T$ shows
     our result.
          \end{proof}

      \begin{example} \label{Ex4:1}
        Consider the two-dimensional discrete system \eqref{eq4:discA}
        with $0 < \rho \le 1$ and $\varphi\in [0,2\pi]$. According to
        \cite[Section 6.1]{BeFrHu20} all four types of first angular values from
        Definition \ref{defangularvalues} coincide and are given
        by
        \begin{equation} \label{eq4:formulatheta1}
          \begin{aligned}
            \theta_1(\varphi,\rho)& = \sup_{x \in \T}g_{\infty}(x,\varphi,\rho), \\
            g_{\infty}(x,\varphi,\rho)&= \begin{cases}
              \frac{1}{2 \pi} \int_{\T} g(D_{\rho}\xi,\varphi,\rho)
              d\mu_1(\xi),
              & \varphi \notin \pi \Q, \\
              \frac{1}{q} \sum_{j=0}^{q-1}g(D_{\rho}T_{\varphi}^j
              D_{\rho}^{-1}x, \varphi, \rho) & \frac{\varphi}{2 \pi} =
              \frac{p}{q},\ 
              p,q \in \N,\ p \perp q,
            \end{cases}\\
            g(x,\varphi,\rho)& = \ang(x,D_{\rho}T_{\varphi}D_{\rho}^{-1}x)
            = \ang(x, A(\rho,\varphi)x).
          \end{aligned}
          \end{equation}
        Recall from Section \ref{sec2.1} that $\ang(x,Ax)$ is the principal
        angle between
        the subspaces $\mathrm{span}(x)$ and $\mathrm{span}(Ax)$ which is the
        minimum of the angles between the vectors $x,Ax$ and $-x,Ax$.
        We refer to \cite[Theorem 6.1]{BeFrHu20} for an explicit evaluation of
        the function $\theta_1(\varphi,\rho)$ and for a picture of its 'hairy graph'. In order to apply Proposition \ref{prop4:baserg} we set
        \begin{align*}
        f(x,\varphi,\rho)= g(D_{\rho}x,\varphi,\rho)= \ang(D_{\rho}x,
        D_{\rho}T_{\varphi} x).
        \end{align*}
        Then \eqref{eq4:formulatheta1} turns into \eqref{eq4:deftheta} and
        we conclude that
        \begin{align*}
          \theta_1(\varphi,\rho) & = \sup_{x\in \T}g_{\infty}(x,\varphi,\rho)=
          \sup_{x \in \T} f_{\infty}(D_{\rho}^{-1}x,\varphi,\rho)= \theta_{\infty}(\varphi,\rho)
        \end{align*}
        is usc w.r.t.\ $(\varphi,\rho)\in [0,2\pi]\times (0,1]$.
        This solves an open problem mentioned after \cite[Theorem 6.1]{BeFrHu20}.
 \end{example}

      \begin{example} \label{Ex4:2r2} (Example \ref{Ex4:2} revisited)\\
        Let us apply Proposition
        \ref{prop4:autonomous} and the formulas
        \eqref{eq4:defgF}, \eqref{eq4:limerg} from its proof
        to compute  the second outer angular values for Example
        \ref{Ex4:2} assuming $\omega_1 \ge \omega_2$:
        \begin{equation} \label{eq4:lims2}
        \begin{aligned}
          \vartheta_2^{\sup,\lim}(A(\omega,\rho)) & =\sup_{v_{1,2} \neq 0}g_{\infty}(v_1,v_2,\omega,\rho) =\sup_{v_{1,2} \neq 0}\lim_{N \to \infty}\frac{1}{N}
          \sum_{j=1}^N g(v_1, F_{\rho_2}^{j-1} v_2,\omega,\rho), \\
          g_{\infty}(v_1,v_2,\omega,\rho)
          & = \begin{cases} \frac{1}{2 \pi} \int_{\T} g(v_1,D_{\rho_2} u_2,\omega,\rho) d\mu_1(u_2) , & \frac{\omega_2}{\omega_1} \notin \Q, \\
              \frac{1}{q} \sum_{j=1}^{q}g(v_1,F_{\rho_2}^{j-1}v_2,\omega,\rho), &
              \frac{\omega_2}{\omega_1}= \frac{p}{q},\ p,q \in \N,\ p \perp q.
          \end{cases}
        \end{aligned}
        \end{equation}
        Here our notation keeps track of the dependence on the parameters $\omega$ and  $\rho$. For every fixed $v_1 \in \T$ we apply Proposition \ref{prop4:baserg}
        with the settings
        \allowdisplaybreaks[3]
        \begin{align*}
          f(v_2,\varphi,\lambda) & = g(v_1,D_{\rho_2} v_2,\omega_1,\tfrac{\omega_1\varphi}{2 \pi},\rho),
            \quad \lambda= (v_1,\omega_1,\rho).
            \end{align*}
        Note that we have for $j \in \N$ by \eqref{eq4:defgF} with $\cL=\{2\}$
        \begin{align*}
          f(T_{\varphi}^{j-1}v_2, \varphi,\lambda)=g(v_1,D_{\rho_2} T_{\varphi}^{j-1}v_2, \omega_1,\tfrac{\omega_1\varphi}{2 \pi},\rho)=g(v_1,F_{\rho_2}^{j-1} D_{\rho_2}v_2, \omega_1,\tfrac{\omega_1\varphi}{2 \pi},\rho).
        \end{align*}
        Therefore, Proposition \ref{prop4:baserg} ensures that
        \begin{equation*} \label{eq4:uscstate}
          \begin{aligned}
                           \theta_{\infty}(\varphi,\lambda)& = \sup_{v_2 \in \T}f_{\infty}(v_2,\varphi,\lambda) = \sup_{v_2 \in \T}g_{\infty}(v_1,D_{\rho_2}v_2,\omega_1,\omega_1 \tfrac{\omega_1\varphi}{2 \pi},\rho)\\
&= \sup_{v_2 \in \T}g_{\infty}(v_1,v_2,\omega_1,
                           \tfrac{\omega_1\varphi}{2 \pi},\rho)
          \end{aligned}
        \end{equation*}
        is usc w.r.t.\ $(\varphi,\lambda) \in [0,2 \pi]\times \T \times (0,\infty)\times (0,1]^2$.
        We still have to maximize with respect to $v_1 \in \T$. For this purpose
        we use the following fact, the proof of which is straightforward. If $f:X\times Y \to \R$ is usc where $X$ is a metric space and $Y$ is a compact metric space,
         then the function defined by $h(x)=\sup_{y \in Y}f(x,y)$
         is usc on $X$. Thus we conclude that
         \begin{align} \label{eq4:expthetasup}
           \vartheta^{\sup,\lim}_2(A(\omega,\rho))& = \sup_{v_1 \in \T}\theta_{\infty}(\tfrac{2 \pi \omega_2}{\omega_1},v_1,\omega_1,\rho)
         \end{align}
         is usc w.r.t.\ $0 < \omega_2\le \omega_1$
         and $0 <  \rho_1,\rho_2 \le 1$. By symmetry one can extend
         this to all $\omega \in (0,\infty)^2$. 

         Finally, we provide an explicit expression for the angular value
          \eqref{eq4:expthetasup}. If $\frac{\omega_2}{\omega_1} \notin \Q$
          then formula \eqref{eq4:sformula} yields
          \begin{equation*} \label{eq4:nonreso}
            \vartheta_2^{\sup,\lim}(A(\omega,\rho))=
            \frac{1}{\pi^2} \int_{[0,\pi]^2} 
            \max\Big(\frac{\rho_1 \omega_1}{\cos^2(\tau_1)+ \rho_1^2 \sin^2(\tau_1)}, \frac{\rho_2 \omega_2}{\cos^2(\tau_2)+ \rho_2^2 \sin^2(\tau_2)}\Big)
            d(\tau_1,\tau_2).
          \end{equation*}
          In the resonant case $\kappa:=\frac{\omega_2}{\omega_1}= \frac{p}{q}$
          with $p,q\in \N$ and $p \perp q$ we return to \eqref{eq4:lims2},
          \eqref{eq4:defgF}, \eqref{eq4:alphadef} with the settings
           $\varphi=2 \pi \kappa$, $\tau_1= \frac{2 \pi}{\omega_1}$, $\ell=1$, $\cL=\{2\}$:
          \begin{align*}
            &\vartheta^{\sup,\lim}_2(A(\omega,\rho))  = \sup_{v_1, v_2 \neq 0}
            \frac{1}{q} \sum_{j=1}^q g(v_1,F_{\rho_2}^{j-1}v_2,\omega,\rho)\\
            &=\sup_{v_1, v_2 \neq 0} \sum_{j=1}^q g(v_1,D_{\rho_2}T_{2 \pi \kappa}^{j-1}D_{\rho_2}^{-1}v_2,\omega,\rho)\\
            & = \sup_{v_1, v_2 \neq 0} \frac{1}{q \tau_1} \sum_{j=1}^q
            \int_0^{\tau_1} \max\big(\alpha(t,v_1,\rho_1,\omega_1),
            \alpha(t,D_{\rho_2}T_{\varphi}^{j-1}D_{\rho_2}^{-1}v_2,\rho_2,\omega_2)\big)dt\\
            & = \sup_{v_1, v_2 \neq 0} \frac{1}{2 \pi q } \sum_{j=1}^q
            \int_0^{2 \pi} \max\big(\alpha(\tfrac{\tau}{\omega_1},v_1,\rho_1,\omega_1),\alpha(\tfrac{\tau}{\omega_1},D_{\rho_2}T_{\varphi}^{j-1}D_{\rho_2}^{-1}v_2,\rho_2,\omega_2)\big) d\tau \\
            & =  \sup_{v_1, v_2 \neq 0} \frac{1}{2 \pi q } \sum_{j=1}^q
          \int_0^{2 \pi}\max \Big( \frac{\rho_1 \omega_1 \|D_{\rho_1}^{-1}v_1\|^2}
                {\|D_{\rho_1} T_{\tau}D_{\rho_1}^{-1}v_1\|^2},
                \frac{\rho_2 \omega_2 \|T_{\varphi}^{j-1}D_{\rho_2}^{-1}v_2\|^2}
           {\|D_{\rho_2} T_{\tau\kappa+(j-1)\varphi}D_{\rho_2}^{-1}
             v_2\|^2}\Big) d\tau.
          \end{align*}
          Using the homogeneity w.r.t.\ $v_1,v_2$ we end up with
          \begin{align*}
          \vartheta^{\sup,\lim}_2(A(\omega,\rho))  & = \sup_{v_1, v_2 \in \T} \frac{1}{2 \pi q } \gamma(v_1,v_2), \\
           \text{where} \quad \gamma(v_1,v_2) & =
           \int_0^{2 \pi}\sum_{j=1}^q \max \Big( \frac{\rho_1 \omega_1 }
                {\|D_{\rho_1} T_{\tau}v_1\|^2},\frac{\rho_2 \omega_2 }
           {\|D_{\rho_2} T_{\tau\kappa+(j-1)\varphi} v_2\|^2}\Big) d\tau.
          \end{align*}
          For the term $\gamma(v_1,v_2)$ we show the relation $\gamma(T_tv_1,v_2)=
          \gamma(v_1,T_{-t \kappa}v_2)$ so that 
          the supremum over $v_1,v_2\in \T$ can be reduced to one dimension.
          For the proof we use that the sets $\{ j\kappa2 \pi=\frac{jp2 \pi}{q} :j=1,\ldots,q\}$ and $\{(j-1)\kappa2 \pi:j=1,\ldots,q\}$
          agree modulo $2 \pi$: \allowdisplaybreaks[3]
          \begin{align*} 
      \gamma(T_tv_1,v_2)& =  \Big\{ \int_0^{2 \pi-t} + \int_{2 \pi-t}^{2 \pi} \Big\}
            \sum_{j=1}^q\max \Big( \frac{\rho_1 \omega_1 }{\|D_{\rho_1} T_{\tau+t}v_1\|^2},
            \frac{\rho_2 \omega_2 }{\|D_{\rho_2}
              T_{\tau\kappa+(j-1)\varphi}v_2\|^2}\Big) d\tau\\
            & = \sum_{j=1}^q \int_t^{2 \pi}\max \Big( \frac{\rho_1 \omega_1 }{\|D_{\rho_1} T_{\sigma}v_1\|^2},
            \frac{\rho_2 \omega_2 }{\|D_{\rho_2}
              T_{\kappa(\sigma-t+(j-1)2\pi)}v_2\|^2}\Big) d\sigma \\
            &\phantom{=\ } +  \int_0^{t}\sum_{j=1}^q\max \Big( \frac{\rho_1 \omega_1 }{\|D_{\rho_1} T_{\sigma+2 \pi}v_1\|^2},
            \frac{\rho_2 \omega_2 }{\|D_{\rho_2}
              T_{\kappa(\sigma-t+j2\pi)}v_2\|^2}\Big) d\sigma\\
            & = \sum_{j=1}^q \int_t^{2 \pi}\max \Big( \frac{\rho_1 \omega_1 }{\|D_{\rho_1} T_{\sigma}v_1\|^2},
            \frac{\rho_2 \omega_2 }{\|D_{\rho_2}
              T_{\kappa(\sigma-t+(j-1)2\pi)}v_2\|^2}\Big) d\sigma \\
            &\phantom{=\ } +  \int_0^{t}\sum_{j=1}^q\max \Big( \frac{\rho_1 \omega_1 }{\|D_{\rho_1} T_{\sigma}v_1\|^2},
            \frac{\rho_2 \omega_2 }{\|D_{\rho_2}
              T_{\kappa(\sigma-t+(j-1)2\pi)}v_2\|^2}\Big) d\sigma\\
            & = \sum_{j=1}^q\int_0^{2 \pi}\max \Big( \frac{\rho_1 \omega_1 }{\|D_{\rho_1} T_{\sigma}v_1\|^2},
            \frac{\rho_2 \omega_2 }{\|D_{\rho_2}
              T_{\kappa(\sigma-t+(j-1)2\pi)}v_2\|^2}\Big) d\sigma\\
            & = \gamma(v_1,T_{-t \kappa}v_2).
          \end{align*}
          Therefore, we  fix $v_2 =e^1= (1,0)^{\top}$ and compute
          the angular value for $\kappa=\frac{p}{q}$ as follows:
          \begin{equation} \label{eq4:thetafinal}
            \begin{aligned}
            &\vartheta^{\sup,\lim}_2(A(\omega,\rho)) =  \sup_{v_1 \in \T}
              \frac{\gamma(v_1,e^1)}{2 \pi q}= \sup_{t \in [0,2\pi]}
              L(t),  \\
             L(t)&= \frac{1}{2 \pi q}\int_0^{2 \pi} \sum_{j=1}^q 
            \max\Big(\tfrac{\rho_1 \omega_1}{\cos^2(t+\tau)+ \rho_1^2 \sin^2(t+\tau)},
            \tfrac{\rho_2 \omega_2}{\cos^2\big( \kappa(\tau+ 2\pi (j-1))\big)
            + \rho_2^2\sin^2\big( \kappa(\tau+ 2\pi (j-1))\big)}\Big) d\tau.
 \end{aligned}
          \end{equation}
          Figure \ref{Hairygraph} shows the graph of $\vartheta^{\sup,\lim}_2(A(\omega,\rho))$ as a function of the ratio $\kappa=\frac{\omega_2}{\omega_1}$ and the parameter
          $\rho_2$. At rational values $\kappa = \frac{p}{q}$ the whole
          line $\{L(t);t\in [0,2 \pi]\}$ is drawn (red) for better visibility.
          One should bear in mind that the peaks of these lines (black) together
          with the smooth surface (blue) at irrational values of $\kappa$ form the graph of a function which
          is upper but not lower semicontinuous.
\begin{figure}[hbt] 
  \begin{center}
   \includegraphics[width=0.95\textwidth]{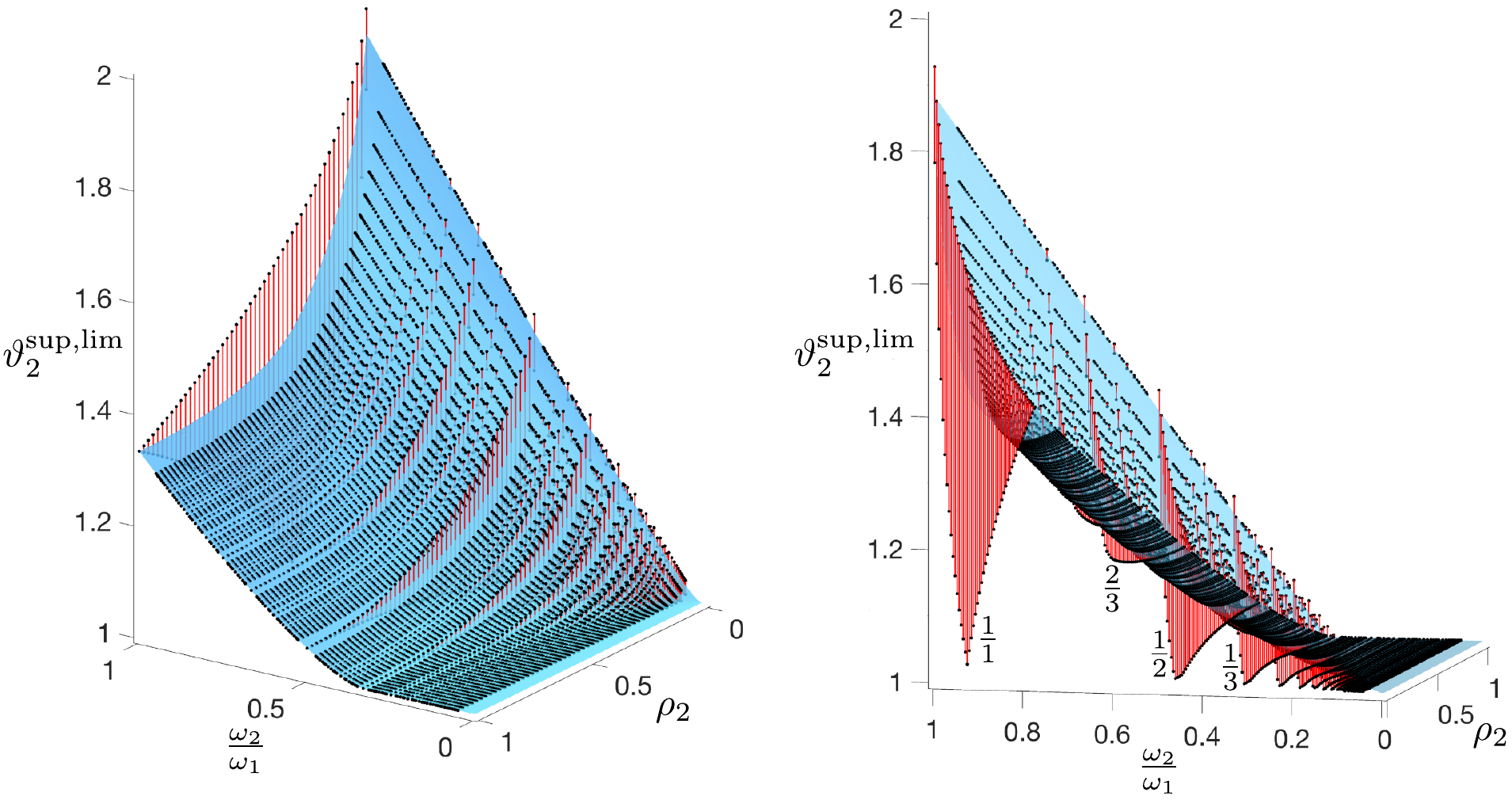}
    \end{center}
  \caption{\label{Hairygraph}Graph of angular values $\vartheta_2^{\sup,\lim}(A(\omega_1,\omega_2,\rho_1,\rho_2))$ for the system 
 \eqref{eq4:autonomous} with  matrix from \eqref{eq4:A4}
    as a function of $\kappa= \frac{\omega_2}{\omega_1},\rho_2\in (0,1]$ for
  $\rho_1= \frac{1}{3}$, $\omega_1=1$. At rational values $\kappa=\frac{p}{q}$
  with $1 \le p \le q \le 20$, $p\perp q$ the vertical
  line $\{L(t):t \in [0,2 \pi]\}$  from  \eqref{eq4:thetafinal} is drawn with angular value at the top.}
\end{figure}

 \end{example}
    

\section*{Acknowledgments}
Both authors are grateful to the Research Centre for Mathematical
Modelling ($\text{RCM}^2$) at Bielefeld University for continuous
support of their joint research.
The work of WJB was funded by the Deutsche Forschungsgemeinschaft
(DFG, German Research Foundation) -- SFB 1283/2 2021 -- 317210226,
and TH thanks the Faculty of Mathematics at Bielefeld
University.


\bibliographystyle{abbrv}

\end{document}